\documentclass[11.5pt,a4paper]{article}
\usepackage[top=1.25in, bottom=1.0in, left=1.0in, right=1.0in]{geometry}

\usepackage[T1]{fontenc}
\usepackage[utf8]{inputenc}
\usepackage{authblk}
\usepackage{geometry}
\usepackage{amssymb} \usepackage{amsmath}
\numberwithin{equation}{section} %修改公式编号方案
\usepackage{amsthm}
\usepackage{diagbox}
\usepackage{empheq}
\usepackage{mdframed}
\usepackage{multirow}
\usepackage{mathrsfs}
\usepackage{booktabs}
\usepackage{lipsum}
\usepackage{graphicx}
\usepackage{color}
\usepackage{psfrag}
\usepackage{pgfplots}
\usepackage{bm}
\usepackage{float}
\usepackage{stmaryrd} 
\usepackage{subfigure}
\usepackage{url}
\usepackage{tikz}
\usepackage{appendix}

\definecolor{gg}{RGB}{10,150,10}
\usepackage{hyperref}

\newtheorem{thm}{Theorem}[section]
\newtheorem{ass}{Assumption}[section]
\newtheorem{coro}{Corollary}[section]

\newtheorem{rmk}{Remark}[section]
\newtheorem{lem}{Lemma}[section]

\newenvironment{keywords}
{\par\noindent\textbf{Keywords:}}

\graphicspath{{FIGS/}}
\definecolor{ocre}{RGB}{243,102,25}
\definecolor{mygray}{RGB}{243,243,244}
\definecolor{deepGreen}{RGB}{26,111,0}
\definecolor{shallowGreen}{RGB}{235,255,255}
\definecolor{deepBlue}{RGB}{61,124,222}
\definecolor{shallowBlue}{RGB}{235,249,255}
\newtheoremstyle{mytheoremstyle}{3pt}{3pt}{\normalfont}{0cm}{\rmfamily\bfseries}{}{1em}{{\color{black}\thmname{#1}~\thmnumber{#2}}\thmnote{\,--\,#3}}
\newtheoremstyle{myproblemstyle}{3pt}{3pt}{\normalfont}{0cm}{\rmfamily\bfseries}{}{1em}{{\color{black}\thmname{#1}~\thmnumber{#2}}\thmnote{\,--\,#3}}
\theoremstyle{mytheoremstyle}
\newmdtheoremenv[linewidth=1pt,backgroundcolor=shallowGreen,linecolor=deepGreen,leftmargin=0pt,innerleftmargin=20pt,innerrightmargin=20pt,]{theorem}{Theorem}[section]
\theoremstyle{mytheoremstyle}
\newmdtheoremenv[linewidth=1pt,backgroundcolor=shallowBlue,linecolor=deepBlue,leftmargin=0pt,innerleftmargin=20pt,innerrightmargin=20pt,]{definition}{Definition}[section]
\theoremstyle{myproblemstyle}
\newmdtheoremenv[linecolor=black,leftmargin=0pt,innerleftmargin=10pt,innerrightmargin=10pt,]{problem}{Problem}[section]
\usepgfplotslibrary{colorbrewer}
\pgfplotsset{width=8cm,compat=1.9}
\title{\LARGE Energy stable and maximum bound principle preserving schemes for the Allen-Cahn equation based on the Saul'yev methods}
\author[a]{Xuelong Gu}
\author[a]{Yushun Wang}
\author[a]{Wenjun Cai \thanks{Corresponding author: caiwenjun@njnu.edu.cn}}
\affil[a]{Ministry of Education Key Laboratory for NSLSCS \\ Jiangsu Collaborative Innovation Center of Biomedical Functional Materials \\ School of Mathematical Sciences \\ 
Nanjing Normal University, Nanjing, 210023, China.  \\ \vspace{-2cm} }
\date{}

\begin{document}
\maketitle
{\noindent}	 \rule[-10pt]{15.5cm}{0.1em}
\begin{abstract}
The energy dissipation law and maximum bound principle are significant characteristics of the Allen-Chan equation. To preserve discrete counterpart of these properties, the linear part of the target system is usually discretized implicitly, resulting in a large linear or nonlinear system of equations. The Fast Fourier Transform (FFT) algorithm is commonly used to solve the resulting linear or nonlinear systems with computational costs of $\mathcal{O}(M^d log M)$ at each time step, where $M$ is the number of spatial grid points in each direction, and $d$ is the dimension of the problem. Combining the Saul'yev methods and the stabilized technique, we propose and analyze novel first- and second-order numerical schemes for the Allen-Cahn equation in this paper. In contrast to the
traditional methods, the proposed methods can be solved by components, requiring only $\mathcal{O}(M^d)$ computational costs per time step. Additionally, they preserve the maximum bound principle and original energy dissipation law at the discrete level. We also propose rigorous analysis of their consistency and convergence. Numerical experiments are conducted to confirm the theoretical analysis and demonstrate the efficiency of the proposed methods.
\end{abstract}
\begin{small}
	\begin{keywords}
		\centering
		Energy stable, Maximum bound principle, Saul'yev methods, Allen-Cahn equation.
	\end{keywords}
\end{small}
{\noindent}	 \rule[-10pt]{15.5cm}{0.1em}
\section{Introduction}
In this paper, we consider the numerical solution of the Allen-Cahn (AC) equation given by
\begin{equation}\label{ac}
	\begin{aligned}
		 & u_t = \varepsilon^2 \Delta u + f(u), \ \bm{x} \in \Omega, \ t\in(0, T], \\
		 & u(\bm{x}, 0) = u_0(\bm{x}), \ \bm{x} \in \Omega,
	\end{aligned}
\end{equation}
subject to periodic boundary conditions. In \eqref{ac}, $\Omega \subset \mathbb{R}^d$ is a rectangular domain, $u(\bm{x}, t): \overline{\Omega} \times (0, T] \to \mathbb{R}$ is a state variable describing the concentration of a phase in a two-components alloy, $\varepsilon>0$ is a diffusion parameter characterizing the width of the diffuse interface, and $f(u)$ is the nonlinear term. 

%The AC equation was originally introduced by Allen and Cahn in \cite{001} to model the motion of anti-phase boundaries in crystalline solids.

When $f$ is taken as the derivative of a bulk potential (i.e., $f = -F^\prime$), \eqref{ac} can be regarded as an $L^2$ gradient flow to the free energy functional
\begin{equation*}
	E(u) = \int_\Omega \Big( \frac{\varepsilon^2}{2} |\nabla u(\bm{x})|^2 + F(u(\bm{x})) \Big) d\bm{x}.
\end{equation*}
There are two categories of commonly-used bulk potentials: the Ginzburg-Landau double-well potential
\begin{equation}\label{F_double_well}
	F^{(POLY)}(u) = \frac{1}{4} (u^2-1)^2,
\end{equation}
which has extreme values of $u = \pm 1$ representing two separate phases, and the Flory-Huggins potential, which has the logarithmic form \cite{002, 003} as
\begin{equation}\label{F_log}
	F^{(LOG)}(u) = \frac{\theta}{2} \big((1+u) {\rm  ln}(1+u) + (1-u){\rm ln}(1-u)\big) - \frac{\theta_c}{2} u^2,
\end{equation}
where $\theta$ and $\theta_c$ are positive constants representing the absolute and critical temperatures, respectively. In physics, the condition $0 < \theta < \theta_c$ is usually imposed to ensure that $F^{(LOG)}$ has a double-well form with two minima of opposite signs. The AC equation has two significant characteristics. First, its solution decreases the free energy along time, i.e., $\frac{dE}{dt} \leq 0$, which is referred to as the energy dissipation law (EDL). Second, if there exists a positive constant $\beta$, such that $f(\beta) \leq 0 \leq f(-\beta)$,  the solution of \eqref{ac} adheres to the maximum bound principle (MBP) \eqref{ac}. That is to say if the supremum norm of the initial condition $u_0(\bm{x})$ is bounded by a positive constant $\beta$, then the supremum norm of its solution is also bounded by $\beta$ for any $t \in (0, T]$ \cite{du_2021}.

Numerical methods that dissipate free energy at a discrete level are referred to as energy stable schemes \cite{eyre_1998}. It has been noted in \cite{dvd_book} that schemes failing to be energy stable may produce unstable or oscillatory solutions. Consequently, various researchers have dedicated themselves to developing energy stable methods in the past few decades. Classical energy stable algorithms include convex splitting methods \cite{convex_splitting1,eyre_1998,csrk}, stabilized semi-implicit methods \cite{cn_ab, tang_imex}, and exponential time difference methods \cite{du_2019, du_2021}. Recently, there have been developments in invariant energy quadratization (IEQ) methods \cite{ieq1, ieq3, ieq4, ieq2} and scalar auxiliary variable (SAV) methods \cite{sav_shen, sav_shen_siam}, which have facilitated the construction of linearly implicit energy stable methods. Extensions of them can be found in \cite{gsav, ieq_gong, esav_high, esav_kg, sav_li, sav_nlsw}. One of the limitations of the SAV and IEQ methods is the preservation of modified energy. To address this, Lagrange multiplier and supplementary variable methods were proposed in \cite{lag, svm}.

%%%%%%%%%%%%%%%%%%%%%%%%%%%%%%%%%%%%%%%%%%%%%%%%%%%%%%%%%%%%%%%%%%%%%%%%%%%%%%%%%%%%%%%%%%%%%%%%%%%%%%%%%%

The preservation of MBP for \eqref{ac} has also received a growing amount of attention as it can avoid nonphysical solutions, particularly for the Flory-Huggins free energy. Complex values may appear in the numerical solutions if the MBP is not preserved due to logarithm arithmetic. The numerical strategies used to develop the MBP-preserving schemes include the semi-implicit backward difference (BDF) \cite{liao_siam1, tang_imex, xiao_fd} and Crank-Nicolson (CN) methods \cite{hou_cnab, hou_leapfrog, shen_2010}, the exponential time difference methods \cite{du_2019, du_2021}, the integration factor methods \cite{ju_ifrk_high, if_nac, zhang_anm}, and operator splitting methods \cite{li_strang,xiao_split}, and so on. Recently, Ju et al. proposed frameworks for developing MBP-preserving methods based on the SAV approach in \cite{ju_gsav,ju_esav}.

%%%%%%%%%%%%%%%%%%%%%%%%%%%%%%%%%%%%%%%%%%%%%%%%%%%%%%%%%%%%%%%%%%%%%%%%%%%%%%%%%%%%%%%%%%%%%%%%%%%%%%%%%%

In order to guarantee energy dissipation or DMP of numerical schemes, the linear part of \eqref{ac} was usually discretized implicitly, necessitating the solution of large linear or nonlinear systems at each step. Although they can be implemented effectively using the Fast Fourier Transform (FFT) algorithm, the computational cost at each step is $\mathcal{O}(M^d\log{M})$. To further enhance the efficiency of these methods, we specifically focus on Saul'yev methods,  which are based on central difference methods and were introduced in \cite{saulyev_1} to solve the following 1D diffusion problem
\begin{equation}\label{diffusion}
	u_t = u_{xx}.
\end{equation}
When the boundary conditions are set to zero, a typical Saul'yev method for \eqref{diffusion} is given by
\begin{equation}\label{saulyev_lr}
	\frac{u_i^{n+1} - u_i^n}{\tau} = \frac{u_{i+1}^n - u_i^n - u_i^{n+1} + u_{i-1}^{n+1}}{h^2}, 
\end{equation}
and $u_0^n = u_M^n = 0$. (A detailed description of the above symbols will be provided in the subsequent contexts.) At first glance, \eqref{saulyev_lr} is a linearly implicit scheme. In contrast to the BDF or CN discretization, it can be solved by components in the order of increasing $i$ as
\begin{equation*}
	u_i^{n+1} = \frac{1 - \tau/h^2}{1 + \tau/h^2} u_i^n + \frac{\tau}{1 + \tau/h^2} (u_{i-1}^{n+1} + u_{i+1}^n).
\end{equation*} 
Notice that the CN or BDF schemes produce a unique tridiagonal system when used for temporal discretization, which can be solved using the Thomas method \cite{thomas} at a cost of $\mathcal{O}(M)$ per step, rather than the FFT method with a cost of $\mathcal{O}(M\log{M})$. However, the Saul'yev method  \eqref{saulyev_lr} is still the most efficient among them, and for higher-dimensional problems, the linear system associated with the CN or BDF methods is no longer tridiagonal. The Thomas method can not be extended straightforwardly without significant modifications. Despite the effectiveness of the Saul'yev methods,  it should be noted that, to the best of our knowledge, they have not yet been applied to periodic problems. Additionally, it remains unclear whether Saul'yev methods possess additional structure-preserving properties.

In this paper, we develop novel numerical algorithms for solving the AC equation with periodic boundary conditions, which are based on the Saul'yev methods and are capable of preserving both the energy dissipation law and DMP for \eqref{ac}. Our main contributions can be summarized as follows,
\begin{enumerate}
	\item We extend the original Saul'yev methods to periodic problems, which can still be solved by components.
	
%	\item The accuracy of the existence Saul'yev methods are limited to first-order to the best of our knowledge.
%	We enhance the accuracy of the existence Saul'yev methods by utilizing the composition technique \cite{hairer_book}, and the proposed second-order schemes can still be solved by components.
	
	\item For the AC equation, we develop novel first- and second-order schemes based on the Saul'yev methods. All the proposed methods are proved to preserve both original energy dissipation law and MBP of \eqref{ac} at the discrete level. The convergence analysis of the proposed methods are proposed comprehensively.
\end{enumerate}

% Although the Saul'yev methods have several advantages, their applications are primarily limited to linear convection-diffusion problems \cite{adi_saulyev_diffusion, saulyev_splitting, saulyev_numer_math, saulyev_jcp, saulyev_1, saulyev_stability, saulyev_diffusion2d}. Little effort has been devoted to applying them to nonlinear problems \cite{saulyev_nonlinear1, saulyev_ch}. Various factors limit their applications: (i) The spatially decoupled property of \eqref{saulyev_lr} and \eqref{saulyev_rl} would be compromised if periodic problems were incorporated. (ii) The stability results of the Saul'yev methods were predominantly determined through von' Neumann analysis, with no consideration given to energy stability. (iii) Both \eqref{saulyev_lr} and \eqref{saulyev_rl} have only first-order accuracy in time. In this paper, we develop first- and second-order numerical methods for the AC equation with periodic boundary conditions based on the Saul'yev methods. We begin by presenting the first-order explicit stabilized Saul'yev method (\textbf{ESS1}). In conjunction with its adjoint (\textbf{ESS1}-adjoint) method, we derive the second-order (\textbf{SS2} and \textbf{SS2}-adjoint) methods. Both \textbf{ESS1} and \textbf{ESS1}-adjoint methods are energy stable and DMP preserving. Consequently, the composed \textbf{SS2} method also preserves these properties. Furthermore, rigorous analyses of the solvability, consistency and convergence of the given schemes are proposed comprehensively.

%%%%%%%%%%%%%%%%%%%%%%%%%%%%%%%%%%%%%%%%%%%%%%%%%%%%%%%%%%
The rest of this paper is organized as follows. In Section \ref{sec 2}, we briefly introduce the central difference approximation and extend Saul'yev methods to periodic problems. In Section \ref{numerical_medthods}, we propose numerical methods and provide rigorous proofs for their DMP, energy dissipation and solvability. In Section \ref{convergence}, we analyze the convergence of proposed methods by the energy method. Numerical experiments are presented to confirm the theoretical analysis and demonstrate the efficiency of the provided schemes in Section \ref{numerical_experiments}. The last section is concerned with the conclusion.

\section{Decoupled Saul'yev methods for the periodic problems}\label{sec 2}
Without loss of generality, we assume $\Omega = (0, 2\pi)^d$.  Let $M$ be a positive integer. We partition $\Omega$ into a grid with uniform step size of $h = 2\pi/M$, denoted by $\Omega_h = \{ \bm{x}_{\bm{i}} = h \bm{i}| \bm{i} \in \{0,1,\cdots,M-1\}^d \}$. The space of grid functions is denoted $\mathbb{M}_h = \{ v | v = \{ v_{\bm{i}} | v_{\bm{i}} = v(\bm{x}_{\bm{i}}), \bm{x}_{\bm{i}} \in \Omega_h \} \}$.
We denote by $\varDelta_h: \mathbb{M}_h \to \mathbb{M}_h$ the discrete contour part of the Laplacian operator. Specifically, for $d = 1,2$, we have
\begin{equation*}
	\varDelta_h v = 
	\left\lbrace
	\begin{aligned}
		&D_h v, \quad  & d = 1, \\ 
		&D_h v + v D^{\mathrm{T}}_h, \quad  & d = 2, 
	\end{aligned} 
	\right.
\end{equation*}
where $v \in \mathbb{M}_h$, $D_h$ is the differential matrix as shown below:
\begin{equation*}
	D_h =
	\frac{1}{h^2}
	\left(
	\begin{smallmatrix}
			-2 & 1 & \cdots & 0 & 1 \\
			1 & -2 & \cdots & 0 & 0 \\
			\vdots & \vdots &  & \vdots & \vdots \\
			0 & 0  & \cdots & -2 & 1 \\
			1 & 0  & \cdots &  1 & -2
		\end{smallmatrix}\right).
\end{equation*}

	The time grid is uniformly partitioned with a step size of $\tau > 0$, we denote by $t_n = n\tau$ and $N_t = \Big[ \frac{T}{\tau} \Big]$.  Given a time grid function $u^n$, we define
	\begin{equation*}
		\delta_t u^{n + \frac{1}{2}} = \frac{u^{n+1} - u^n}{\tau}, \quad u^{n+\frac{1}{2}} = \frac{u^n + u^{n+1}}{2}.
	\end{equation*}

We then extend spatially decoupled Saul'yev methods to periodic problems using the 1D problem \eqref{diffusion} for illustration. The space-discrete problem of \eqref{diffusion} is to find $u_h \in \mathbb{M}_h$, such that
\begin{equation}\label{semi_1d}
	\frac{d u_h}{dt} = D_h u_h.
\end{equation}
We introduce a decomposition $D_h = D_a + D_b$, where
\begin{equation*}
	D_a =
	\frac{1}{h^2}
	\left(
	\begin{smallmatrix}
			-1 & 1 & \cdots & 0 & 1 \\
			0 & -1 & \cdots & 0 & 0 \\
			\vdots & \vdots &  & \vdots & \vdots \\
			0 & 0  & \cdots & -1 & 1 \\
			0 & 0  & \cdots &  0 & -1
		\end{smallmatrix}\right), \quad
	D_b =
	\frac{1}{h^2}
	\left(
	\begin{smallmatrix}
			-1 & 0 & \cdots & 0 & 0 \\
			1 & -1 & \cdots & 0 & 0 \\
			\vdots & \vdots &  & \vdots & \vdots \\
			0 & 0  & \cdots & -1 & 0 \\
			1 & 0  & \cdots &  1 & -1
		\end{smallmatrix}\right),
\end{equation*}
and recast \eqref{semi_1d} into
\begin{equation}\label{lin_split}
	\frac{d u_h}{dt} = D_a u_h + D_b u_h.
\end{equation}
Discretizing one term on the right-hand side of \eqref{lin_split} explicitly and the other implicitly results in the following fully discrete scheme
\begin{equation}\label{1d_matrix}
    \delta_t u^{n+\frac{1}{2}} = D_a u^n + D_b u^{n+1},
\end{equation}
which can be expressed componentwise as
\begin{equation}\label{1d_point}
    \delta_t u_i^{n+\frac{1}{2}} = \frac{u_{i+1}^{n+\delta_{iM-1}} - u_i^n - u_i^{n+1} + u_{i-1}^{n+1 - \delta_{i0}}}{h^2}.
\end{equation}
Here, $\delta_{ij}$ denotes the Kronecker symbol.

\begin{rmk}
A naive approach to extend \eqref{saulyev_lr} to periodic problems is to directly enforce the boundary condition $u_M^n=u_0^n$. However, this results in a scheme that cannot be solved component-wise. In contrast, \eqref{1d_point} makes slight modifications to the temporal discretization at boundary points in addition to enforcing the boundary conditions, enabling the extended scheme to be solved by components.
\end{rmk}
\begin{rmk}
	In the classical CN and BDF methods, the right-hand side of \eqref{semi_1d} is discretized at the same time level, specifically at $t_{n+1/2}$ and $t_{n+1}$, respectively. Conversely, the presented Saul'yev methods only implicitly discretize the lower triangular part of \eqref{semi_1d}, meaning they only require the solution of triangular systems instead of tridiagonal systems at each step. This fundamental difference allows the proposed Saul'yev methods to be solved in components.
\end{rmk}
	While applying to the AC equation, the proposed scheme has numerous other satisfactory advantages.
	It preserves both original discrete EDL (DEDL) and DMP, which will be demonstrated in detail in the subsequent contexts.
\section{Fully discrete Saul'yev methods for the AC equation}\label{numerical_medthods}
\subsection{Space-discrete problem}\label{space_stabilization}

Without loss of generality, we investigate the 2D AC equation with periodic boundary conditions in the domain $\Omega = (0, 2\pi)^2$. Extensions to 1D, 3D problems are straightforward.
 We introduce some useful notations.
Given $u, v \in \mathbb{M}_h$, we define the discrete $L^2$ inner product, discrete $L^2$ and $L^\infty$ norms as follows,
\begin{equation*}
	\langle v, w \rangle = h^2 \sum\limits_{i = 0}^{M-1}\sum\limits_{j = 0}^{M-1}  v_{ij}w_{ij}, \quad \ \|v\| = \sqrt{\langle v, v \rangle}, \quad \|v\|_\infty = \max\limits_{0 \leq i \leq M-1 \atop 0 \leq j \leq M-1} |v_{ij}|.
\end{equation*}
Additionally, for $\bm{v} = (v^1, v^2), \bm{w} = (w^1, w^2) \in \mathbb{M}_h \times \mathbb{M}_h$, we define
\begin{equation*}
	\langle \bm{v}, \bm{w} \rangle = \langle v^1, w^1 \rangle + \langle v^2, w^2 \rangle, \quad \|\bm{v}\| = \sqrt{\langle \bm{v}, \bm{v} \rangle}.
\end{equation*}
We define the discrete gradient operator
\begin{equation*}
	\nabla_h v_{ij} = \left( \frac{v_{i+1 j} - v_{ij}}{h}, \frac{v_{ij+1} - v_{ij}}{h}\right)^\mathrm{T}, \quad 0 \leq i,j \leq M-1.
\end{equation*}
By employing periodic boundary conditions, it is readily to confirm the following summation-by-parts formula
\begin{equation} \label{summation_by_parts}
	\langle v, \varDelta_h w \rangle = - \langle \nabla_h v, \nabla_h w \rangle = \langle \varDelta_h v, w \rangle.
\end{equation}
In this paper, we specify $u_e(x, y, t)$ be the exact solution of \eqref{ac}, the approximation of which at $(x_j, y_k, t_n)$ is denoted $u_{jk}^n$.

Since we aim at developing energy stable and DMP preserving methods, it is necessary to make several assumptions about $f(\cdot)$.
\begin{ass}[\cite{du_2021}]
	Assume that $f(\cdot)$ is continuously differentiable in the finite domain $[-\beta, \beta]$ such that the supremum norm of its derivative $\|f^\prime\|_{C[-\beta, \beta]} = \max\limits_{\xi \in [-\beta, \beta]} |f^\prime(\xi)|$ is bounded.
\end{ass}
\begin{ass}[\cite{du_2021}]\label{sign_rule}
	Suppose there exists a constant $\beta > 0$, such that
	\begin{equation}
		f(\beta) \leq 0 \leq	f(-\beta).
	\end{equation}
\end{ass}
\begin{rmk}
	We remark here that the nonlinear functions corresponding to the bulk energy \eqref{F_double_well} and \eqref{F_log} satisfy Assumption \ref{sign_rule}. Specifically, the nonlinear function related to the double-well potential is 
	\begin{equation}\label{f_double_well}
		f^{(POLY)}(u) = u - u^3,
	\end{equation}
	which fulfills Assumption \ref{sign_rule} in the sense that $f^{(POLY)}(1) = 0 = f^{(POLY)}(-1)$.

	Likewise, the nonlinear function associated with the Flory-Huggins potential can be written as
	\begin{equation}\label{f_log}
		f^{(LOG)}(u) = \frac{\theta}{2} {\rm ln} \frac{1-u}{1+u} + \theta_c u.
	\end{equation}
	Let $\rho \in (\sqrt{1 - \frac{\theta}{2\theta_c - \theta}}, 1)$ be the positive root of $f^{(LOG)}(\xi) = 0$. It can be confirmed that Assumption \ref{sign_rule} holds since $f^{(LOG)}(-\rho) = 0 = f^{(LOG)}(\rho)$.
\end{rmk}

We then introduce the space-discrete problem for \eqref{ac} and then briefly review the stabilized technique, which was applied to construct energy stable or DMP preserving methods in \cite{du_2019,du_2021,ju_esav,shen_2010,tang_imex}.

The space-discrete problem for \eqref{ac} is to find $u_h(t) \in \mathbb{M}_h$, such that
\begin{equation}\label{ac_semi}
	\frac{d u_h(t)}{dt} = \varepsilon^2 \varDelta_h u_h(t) + f(u_h(t)).
\end{equation}
By taking the discrete inner-product on both sides of \eqref{ac_semi} with $- \frac{d u_h(t)}{dt}$, it can be shown that \eqref{ac_semi} satisfies the following DEDL:
\begin{equation*}\label{dissipation_semi}
	\frac{d}{dt} E_h(u_h(t)) = - \Big\| \frac{d u_h(t)}{dt} \Big\|^2 \leq 0,
\end{equation*}
where $E_h(\cdot)$ is the discrete energy functional defined as
\begin{equation*}
	E_h(v) = \frac{\varepsilon^2}{2} \|\nabla_h v\|^2 + \langle F(v), 1 \rangle, \quad \forall v \in \mathbb{M}_h.
\end{equation*}
\eqref{ac_semi} additionally preserves the DMP, i.e., $\|u_h(0)\|_\infty \leq \beta$ implies $\|u_h(t)\|_\infty \leq \beta$ for any $t \in (0, T]$. A detailed derivation can be found in \cite{mbp_fd}.

%As an example, a first-order stabilized semi-implicit (\textbf{SSI1}) method is
%\begin{equation}\label{fully_cn_stable}
%	\delta_t u^{n+\frac{1}{2}} = \varepsilon^2 \varDelta_h u^{n+1} -  \kappa u^{n+1} + \kappa u^n + f(u^n).
%\end{equation}
%The following Lemma \ref{convex_splitting} motivates the construction of \eqref{fully_cn_stable}.

%By taking the discrete inner product on both sides of \eqref{fully_cn_stable} with $\tau \delta_t u^{n+\frac{1}{2}}$, using \eqref{nonlinear_difference}, and performing some simple calculations, it is readily to confirm that \eqref{fully_cn_stable} is energy stable.
%\begin{rmk}
%		The above argument contains a flaw in that \eqref{xi_geq} holds only when $I$ is a bounded domain or when $f(\cdot)$ is Lipschitz continuous. However, the $L^\infty$ boundedness of the numerical solutions is generally not guaranteed, and the nonlinear terms \eqref{f_double_well} and \eqref{f_log} we are interested in are not Lipschitz continuous. Therefore, \eqref{nonlinear_difference} can not be derived from Lemma \ref{convex_splitting}, since $F_e(\xi)$ in \eqref{cs_bulk} is not convex. However, the condition \eqref{xi_geq} holds as long as the DMP of the AC equation can be preserved, which is also a prerequisite for obtaining the DEDL of \eqref{fully_cn_stable} from \eqref{nonlinear_difference}. This is another justification for developing DMP preservation methods. It is worth mentioning that the DMP and the DEDL of the scheme \eqref{fully_cn_stable} have been proved in \cite{tang_imex}.
%\end{rmk}

\subsection{Time stepping methods}

To construct energy stable or DMP preserving schemes, a stabilized parameter is introduced to rewrite \eqref{ac_semi} as
\begin{equation}\label{semi_stable}
	\dfrac{d u_h(t)}{dt} = \varepsilon^2 \varDelta_h u_h(t) - \kappa  u_h(t)  + \kappa u_h(t) + f(u_h(t)).
\end{equation}
We denote by
\begin{equation*}
	\varDelta_a u = D_a u + u D_a^\mathrm{T}, \quad \varDelta_b u = D_b u + u D_b^\mathrm{T}.
\end{equation*}
Employing the Saul'yev methods to \eqref{semi_stable}, the fully discrete first-order explicit stabilized Saul'yev (\textbf{ESS1}) method for \eqref{ac} is as follows,
\begin{equation}\label{ess1}
	\delta_t u^{n+\frac{1}{2}} = \varepsilon^2 (\varDelta_a u^{n} + \varDelta_b u^{n+1} ) + f(u^n ) - \kappa (u^{n+1} - u^{n}).
\end{equation}
It can be observed from the definitions of $\varDelta_a$ and $\varDelta_b$ that \eqref{ess1} admits the following pointwise implementation
\begin{equation}\label{ess1_form2}
	\begin{aligned}
		u_{ij}^{n+1} & = \frac{1 +  \tau (\kappa - 2 r)}{1 + \tau (\kappa + 2  r)} u_{ij}^n +  \frac{\tau}{1 + \tau (\kappa + 2 r)} f(u_{ij}^n)                                                                 \\
		             & + \frac{r}{1 + \tau (\kappa + 2 r)} (u_{i-1 j}^{n+1-\delta_{i0}} + u_{i+1 j}^{n+\delta_{iM-1}} + u_{i j-1}^{n+1-\delta_{j0}} + u_{i j+1}^{n+\delta_{jM-1}} ), \quad  0\leq i,j \leq M-1, \\
	\end{aligned}
\end{equation}
where $r = \varepsilon^2 / h^2$. Notably, the \textbf{ESS1} method computes the solution in the direction of increasing $i,j$, which is illustrated in the left of Figure \ref{computation_order}(a) for clarity.

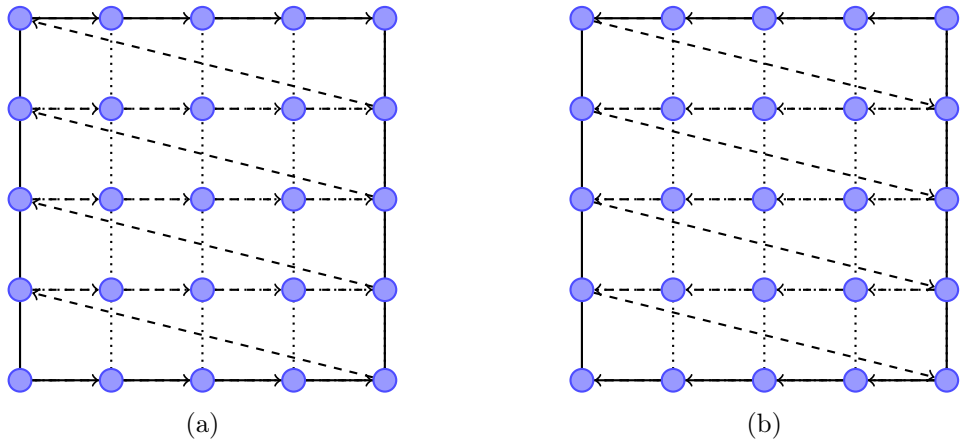
\begin{figure}[H]
	\centering
	\begin{minipage}{0.3\textwidth}
		
		\begin{tikzpicture}[scale=0.6,inner sep=0pt,minimum size=3mm,thick]
			
			\draw[solid] (0,0) -- (8,0)--(8,-8)--(0,-8)--cycle;
			
			\foreach \x in {-8,-6,-4,-2}
			\draw[dotted] (0,\x)--(8,\x);
			
			\foreach \x in {2,4,6,8}
			\draw[dotted] (\x,0)--(\x,-8);
			
			\foreach \x in {1,2,3,4,5}
			\foreach \y in {1,2,3,4,5}
			\node(\x\y) at (2*\y-2,2*\x-10) [circle,draw=blue!70,fill=blue!40] {};
			
			\path[->,dashed] (11) edge (12); 
			\path[->,dashed] (12) edge (13); 
			\path[->,dashed] (13) edge (14);
			\path[->,dashed] (14) edge (15); 
			
			\path[->,dashed] (15) edge (21); 
			\path[->,dashed] (21) edge (22); 
			\path[->,dashed] (22) edge (23); 
			\path[->,dashed] (23) edge (24); 
			\path[->,dashed] (24) edge (25); 
			
			\path[->,dashed] (25) edge (31); 
			\path[->,dashed] (31) edge (32); 
			\path[->,dashed] (32) edge (33); 
			\path[->,dashed] (33) edge (34); 
			\path[->,dashed] (34) edge (35); 
			
			\path[->,dashed] (35) edge (41); 
			\path[->,dashed] (41) edge (42); 
			\path[->,dashed] (42) edge (43); 
			\path[->,dashed] (43) edge (44); 
			\path[->,dashed] (44) edge (45); 
			
			\path[->,dashed] (45) edge (51); 
			\path[->,dashed] (51) edge (52); 
			\path[->,dashed] (52) edge (53); 
			\path[->,dashed] (53) edge (54); 
			\path[->,dashed] (54) edge (55); 
			
			\node at (4,-9) {(a)};
			
		\end{tikzpicture}
		
	\end{minipage}\hspace{2.5cm}
	\begin{minipage}{0.3\textwidth}
		\centering
		
		\begin{tikzpicture}[scale=0.6,inner sep=0pt,minimum size=3mm,thick]
			
			\draw[solid] (0,0) -- (8,0)--(8,-8)--(0,-8)--cycle;
			
			\foreach \x in {-8,-6,-4,-2}
			\draw[dotted] (0,\x)--(8,\x);
			
			\foreach \x in {2,4,6,8}
			\draw[dotted] (\x,0)--(\x,-8);
			
			\foreach \x in {1,2,3,4,5}
			\foreach \y in {1,2,3,4,5}
			\node(\x\y) at (2*\y-2,2*\x-10) [circle,draw=blue!70,fill=blue!40] {};
			
			\path[->,dashed] (12) edge (11); 
			\path[->,dashed] (13) edge (12); 
			\path[->,dashed] (14) edge (13);
			\path[->,dashed] (15) edge (14); 
			
			\path[->,dashed] (21) edge (15); 
			\path[->,dashed] (22) edge (21); 
			\path[->,dashed] (23) edge (22); 
			\path[->,dashed] (24) edge (23); 
			\path[->,dashed] (25) edge (24); 
			
			\path[->,dashed] (31) edge (25); 
			\path[->,dashed] (32) edge (31); 
			\path[->,dashed] (33) edge (32); 
			\path[->,dashed] (34) edge (33); 
			\path[->,dashed] (35) edge (34); 
			
			\path[->,dashed] (41) edge (35); 
			\path[->,dashed] (42) edge (41); 
			\path[->,dashed] (43) edge (42); 
			\path[->,dashed] (44) edge (43); 
			\path[->,dashed] (45) edge (44); 
			
			\path[->,dashed] (51) edge (45); 
			\path[->,dashed] (52) edge (51); 
			\path[->,dashed] (53) edge (52); 
			\path[->,dashed] (54) edge (53); 
			\path[->,dashed] (55) edge (54); 
			
			\node at (4,-9) {(b)};
			
		\end{tikzpicture}
	\end{minipage}

	\caption{Computation order of the \textbf{ESS1} (a) and \textbf{ESS1-adjoint} (b) methods.}
	\label{computation_order}
\end{figure}

%\begin{figure}[H]
%	\centering
%	
%	\subfigure[\label{ess1_sub}]{ 
%		\includegraphics[width = 0.32\linewidth]{update_ess1}
%	}
%		\subfigure[\label{ess1_adj_sub}]{ 
%		\includegraphics[width = 0.32\linewidth]{update_ess1_adj}
%	}
%
%	\caption{Computation order of the \textbf{ESS1} \subref{ess1_sub} and \textbf{ESS1-adjoint} \subref{ess1_adj_sub} methods.}
%	\label{computation_order}
%\end{figure}

In practical implementation, second-order schemes are preferable. It is effective to extend the \textbf{ESS1} method to second-order by composing it with its adjoint method, as described in \cite{hairer_book}. Specifically, if we denote the propagator of \eqref{ess1} by

%For instance, one can use the second-order BDF or CN methods to discretize the linear component of \eqref{ac}, while employing extrapolation techniques to discretize its nonlinear term. This approach can significantly enhance the precision of the schemes. 
%However, we have decided not to utilize these methods in our implementation, as it may compromise the benefits of a pointwise implementation.

\begin{equation}\label{evolution_ess1}
	u^{n+1} = \mathcal{S}(\tau) u^n.
\end{equation}
Its adjoint method, denoted by $u^{n+1} = \widetilde{\mathcal{S}}(\tau)u^n = \mathcal{S}^{-1}(-\tau)u^n$, can be obtained by exchanging $\tau \leftrightarrow -\tau$ and $n \leftrightarrow n+1$ in \eqref{ess1}. The resulting \textbf{ESS1-adjoint} method is
\begin{equation}\label{ess1_adjoint}
	\delta_t u^{n+\frac{1}{2}} = \varepsilon^2 (\varDelta_a u^{n+1} + \varDelta_b u^n) + f(u^{n+1}) + \kappa(u^{n+1} - u^n).
\end{equation}
Similarly, \eqref{ess1_adjoint} is also compatible with a pointwise reformulation, that is
\begin{equation}\label{ess1_adjoint_form2}
	\begin{aligned}
		 & \left(-1 + \tau(\kappa - 2r) \right)u_{ij}^{n+1} + \tau f(u_{ij}^{n+1}) + (1 - \tau (\kappa + 2r) )u_{ij}^n                                                             \\
		 & \quad + \tau r( u_{i-1 j}^{n+1-\delta_{i0}} + u_{i+1 j}^{n+\delta_{iM-1}} + u_{i j-1}^{n+1-\delta_{j0}} + u_{i j+1}^{n+\delta_{jM-1}}  ) = 0,	\quad 0 \leq i,j \leq M-1.
	\end{aligned}
\end{equation}
In contrast to the \textbf{ESS1} method, the \textbf{ESS1-adjoint} method needs to be solved in the direction of decreasing $i, j$. The computational order is also displayed in Figure \ref{computation_order}(b). 
 
It should be noted that although \eqref{ess1_adjoint_form2} is a fully implicit scheme, it can still be efficiently solved using the Newton's method, as it involves only scalar nonlinear equations. Our numerical experiments have shown that the efficiency of \textbf{ESS1-adjoint} method are   comparable to many other existing methods, despite being fully implicit.

The second-order \textbf{SS2} method is then obtained by composing $\mathcal{S}(\tau)$ and $\widetilde{\mathcal{S}}(\tau)$ as
\begin{equation}\label{ss2}
	u^{n+1} = \widetilde{\mathcal{S}}(\tau/2) \mathcal{S}(\tau/2) u^n.
\end{equation}
Analogously, the (\textbf{SS2-adjoint}) method is
\begin{equation}
	u^{n+1} = \mathcal{S}(\tau/2) \widetilde{\mathcal{S}}(\tau/2) u^n.
\end{equation}
It is worth mentioning that higher-order schemes can also be obtained by incrementing the composition stages, and readers are referred to \cite{hairer_book} for details. This paper only focus on the first- and second-order methods.

\subsection{DMP of the proposed methods}
\begin{lem}\label{bound_nonlinear}
    Let $\phi(\xi)$ be defined as
	\begin{equation*}
		\phi(\xi) = \frac{1 + \tau (\kappa - 2r)}{1 + \tau(\kappa + 2r)} \xi + \frac{\tau}{1 + \tau(\kappa + 2r)} f(\xi), \quad \xi \in [-\beta, \beta].
	\end{equation*}
	Under the conditions $\kappa \geq \|f^\prime\|_{C[-\beta, \beta]}$ and $0 < \tau  \leq \frac{h^2}{2\varepsilon^2}$, it holds that
	\begin{equation*}
		\|\phi\|_{C[-\beta, \beta]} \leq  \beta.
	\end{equation*}
\end{lem}
\begin{proof}
	Differentiating both sides of $\phi(\xi)$ gives
	\begin{equation*}
		\phi^\prime(\xi) = \frac{1 + \tau (\kappa + f^\prime(\xi) - 2r)}{1 + \tau (\kappa + 2r)} \geq 0,
	\end{equation*}
	which implies that $\phi(\xi)$ is monotonically increasing on $[-\beta, \beta]$. Assumption \eqref{sign_rule} then produces
	\begin{equation*}
		\|\phi\|_{C[-\beta, \beta]} = \max \{|\phi(-\beta)|, |\phi(\beta)| \} \leq \beta.
	\end{equation*}
	The proof is thus completed.
\end{proof}
\begin{thm}[DMP of \textbf{ESS1}]\label{mbp_ess1}
	Assume that the initial condition $u^0$ satisfies $\|u^0\|_\infty \leq \beta$. Then \textbf{ESS1} method preserves the DMP in the sense
	\begin{equation*}
		\|u^n\|_\infty \leq \beta, \quad \forall \ 0 \leq n \leq N_t.
	\end{equation*}
	 Provided that the stabilized parameter $\kappa \geq \|f^\prime\|_{C[-\beta, \beta]}$, and the time step $0 < \tau \leq \frac{h^2}{2 \varepsilon^2}$.
\end{thm}
\begin{proof}
    We will use an inductive argument to demonstrate the proof. By the hypotheses of Theorem \ref{mbp_ess1}, the conclusion is true for $n=0$. Assuming that it holds for $n=l$, we will show that it is also true for $n=l+1$ by utilizing an alternative inductive argument along the ascending order of indices $i$ and $j$. Combining the induction with Lemma \ref{bound_nonlinear}, we obtain from \eqref{ess1_form2} that
	\begin{equation*}
		\begin{aligned}
			|u_{0,0}^{l+1}| & \leq \frac{1 + \tau (\kappa - 2r)}{1 + \tau (\kappa + 2r)} \beta + \frac{r}{1 + \tau (\kappa + 2r)} (|u_{M-1, 0}^l| + |u_{1,0}^l| + |u_{0, M-1}^l| + |u_{1, 0}^{l}|) \\
			               & \leq \frac{1 + \tau (\kappa - 2r)}{1 + \tau (\kappa + 2r)} \beta + \frac{4r}{1 + \tau (\kappa + 2r)} \beta \leq \beta.
		\end{aligned}
	\end{equation*}
	Assuming $|u_{i-1j}^{l+1}| \leq \beta$ and $|u_{ij-1}^{l+1}| \leq \beta$ for $0\leq i,j \leq M-1$, a comparable process gives:
	\begin{equation*}
		|u_{ij}^{l+1}| \leq \frac{1 + \tau (\kappa - 2r)}{1 + \tau (\kappa + 2r)} \beta + \frac{4r}{1 + \tau (\kappa + 2r)} \beta \leq \beta.
	\end{equation*}
	Therefore, $|u_{ij}^{l+1}| \leq \beta$ for all $0 < i,j \leq M-1$, and the result is also true for $n = l+1$. This completes the proof.
\end{proof}
To obtain the solvability and DMP of the \textbf{ESS1-adjoint} method, it is necessary to make an additional Assumption \ref{assumption_strong} to the nonlinear term.
\begin{ass}\label{assumption_strong}
	Suppose that the nonlinear term $f(\cdot)$ is an odd function and there exists a constant $\beta > 0$, such that $f(\beta) = 0$.
\end{ass}
It is worth noting that Assumption \ref{assumption_strong} implies Assumption \ref{sign_rule}. Furthermore, both nonlinear terms $f^{(POLY)}$ and $f^{(LOG)}$   \eqref{f_double_well} and \eqref{f_log} satisfy Assumption \ref{assumption_strong}.
\begin{thm}\label{mbp_ess1_adjoint}
	Assume that the initial condition $u^0$ satisfies $\|u^0\|_\infty \leq \beta$. Then the \textbf{ESS1-adjoint} method \eqref{ess1_adjoint_form2} is uniquely solvable in $[-\beta, \beta]$. Consequently, the solutions generated by \textbf{ESS1-adjoint} satisfy
	\begin{equation*}
		\|u^n\|_\infty \leq \beta, \quad 0 \leq n \leq N_t.
	\end{equation*}
	Provided that the time step $0 < \tau \leq \min\{\frac{h^2}{\kappa h^2 + 2\varepsilon^2},  \frac{1}{\kappa + \|f^\prime\|_{C[-\beta, \beta]} }\}$.
\end{thm}
\begin{proof}
	We use mathematical induction to prove that \eqref{ess1_adjoint_form2} has a unique solution within $[-\beta,\beta]$. The assumption holds true for $n=0$. Assuming it is true for $n=l$, we prove that it is also true for $n=l+1$ using an additional induction argument in decreasing order of indices $i$ and $j$. Since procedure of getting the boundedness of $|u_{M-1,M-1}^{l+1}|$ will be similar to the subsequent derivations, we will only show that $|u_{i+1,j}^{l+1}| \leq \beta$ and $|u_{i,j+1}^{l+1}|\leq \beta$ imply $|u_{ij}^{l+1}|\leq \beta$. Let us define 
	\begin{equation*}
		\varphi(\xi) = (-1 + \tau (\kappa - 2r)) \xi + \tau f(\xi) + \eta,
	\end{equation*}
	where
	\begin{equation*}
		\eta = (1 - \tau (\kappa + 2r) )u_{ij}^l + \tau r( u_{i-1, j}^{l+1-\delta_{i0}} + u_{i+1, j}^{l+\delta_{iM-1}} + u_{i, j-1}^{l+1-\delta_{j0}} + u_{i, j+1}^{l+\delta_{jM-1}}  ).
	\end{equation*}
	We prove the solvability of $\varphi(\xi) = 0$ by showing that $\varphi(-\beta)$ and $\varphi(\beta)$ have opposite signs. Assumption \ref{assumption_strong} implies that
	\begin{equation}\label{mbp_ess1_adjoint_eq1}
		\varphi(\beta)\varphi(-\beta) = \Big(\eta + (-1 + \tau (\kappa - 2r)) \beta \Big) \Big(  \eta - (-1 + \tau (\kappa - 2r)) \beta \Big).
	\end{equation}
	Calculating the ratio of $|\eta|$ to $\beta|-1 + \tau (\kappa - 2r)|$ yields
	\begin{equation}\label{mbp_ess1_asjoint_eq2}
		\frac{|\eta|}{\beta|-1 + \tau (\kappa - 2r)|} \leq \frac{1 -  \tau\kappa - 2\tau r+4\tau r}{1 - \tau \kappa + 2 \tau r} \leq 1.
	\end{equation}
	Note that the time step restriction implies $1 - \tau \kappa \pm 2\tau r > 0$. Combining \eqref{mbp_ess1_adjoint_eq1} and \eqref{mbp_ess1_asjoint_eq2} gives $\varphi(-\beta)\varphi(\beta) \leq 0$. Therefore, there is at least one solution to \eqref{ess1_adjoint_form2} in $[-\beta, \beta]$ according to the existence theorem of zero points. 
	
	If there exists distinct $\xi_1$ and $\xi_2$ such that
	\begin{equation*}
		\begin{aligned}
			 & (-1 + \tau (\kappa - 2r)) \xi_1 + \tau f(\xi_1) + \eta = 0, \\
			 & (-1 + \tau (\kappa - 2r)) \xi_2 + \tau f(\xi_2) + \eta = 0.
		\end{aligned}
	\end{equation*}
	Making the difference between above two equations, we obtain
	\begin{equation*}
		\left|(1 + 2 r \tau ) (\xi_1 - \xi_2) \right| \leq (\tau \kappa + \tau\|f^\prime\|_{C[-\beta, \beta]}) |\xi_1 - \xi_2|,
	\end{equation*}
	that is
	\begin{equation*}
		\left( 1 + 2r\tau - \tau\kappa - \tau \|f^\prime\|_{C[-\beta, \beta]} \right) |\xi_1 - \xi_2| \leq 0.
	\end{equation*}
	By the condition $\tau \leq \frac{1}{\kappa + \|f^\prime\|_{C[-\beta, \beta]}}$, we can infer that $\xi_1 = \xi_2$. Consequently, the nonlinear system \eqref{mbp_ess1_asjoint_eq2} has a unique solution within $[-\beta, \beta]$. The result holds for $u_{ij}^{l+1} \ \forall \ 0\leq i,j \leq M-1$, and the proof is thus completed.
\end{proof}
\begin{rmk}
	We have demonstrated that the nonlinear system \eqref{ess1_adjoint_form2} has a unique solution within the subset $\{\xi \in \mathbb{R} : |\xi| \leq \beta\}$. In practical implementation, we will employ the Newton's method as a nonlinear solver for \eqref{ess1_adjoint_form2}, with an initial guess of $|\xi_0| \leq \beta$ to ensure the iteration converges to the desired solution.
\end{rmk}
\begin{rmk}\label{nonlinear_double_well}
	When considering the AC equation with double-well potential, the above proof can be simplified. Specifically, the nonlinear system \eqref{ess1_adjoint_form2} can be reduced to the following cubic equation under this circumstance
	\begin{equation}\label{ess1_adjoint_cubic}
		\xi^3 + p \xi + q = 0,
	\end{equation}
	where
	\begin{equation*}
		\begin{aligned}
			p & = \frac{1}{\tau} + \frac{2 \varepsilon^2}{h^2} - \kappa - 1,                                                                                                                                                             \\
			q & = -(\frac{1}{\tau} - \frac{2\varepsilon^2}{h^2} - \kappa)u_{ij}^n - \frac{\varepsilon^2}{h^2} ( u_{i-1, j}^{n+1-\delta_{i0}} + u_{i+1, j}^{n+\delta_{iM-1}} + u_{i, j-1}^{n+1-\delta_{j0}} + u_{i, j+1}^{n+\delta_{jM-1}} ).
		\end{aligned}
	\end{equation*}
	If the times step satisfies $\tau \leq \frac{1}{1+\kappa}$, the unique real solution of \eqref{ess1_adjoint_cubic} can be given by 
	\begin{equation}\label{root_cubic}
		\xi = \sqrt[3]{-\frac{q}{2} + \sqrt{\varDelta}} + \sqrt[3]{-\frac{q}{2} - \sqrt{\varDelta}}.
	\end{equation}
	where $\varDelta = \frac{q^2}{4} + \frac{p^3}{27}$ is the discriminant of \eqref{ess1_adjoint_cubic}, as determined by the Cardano's formula. However, using \eqref{root_cubic} may be more expensive in practical computations, and numerical experiments will compare this method to Newton's method in detail. 
\end{rmk}
\begin{thm}[DMP of \textbf{SS2} method] \label{dmp_ss2}
	Assume that the initial value $u^0$ is bounded in the sense $\|u^0\|_\infty \leq \beta$. Then the solution generated by \textbf{SS2} satisfies
	\begin{equation*}
		\|u^n\|_{\infty} \leq \beta, \quad  0 \leq n \leq N_t.
	\end{equation*}
	Provided that the time step $0 < \tau \leq \min\{ \frac{2h^2}{\kappa h^2 + 2\varepsilon^2}, \frac{2}{\kappa + \|f^\prime\|_{C[-\beta, \beta]}}  \}$.
\end{thm}
\begin{proof}
	We also prove this theorem by mathematical induction. The result is clear true for $n=0$. Suppose that it is true for $n = l$. We define
	\begin{equation*}
		\begin{aligned}
			v^l = \mathcal{S}(\tau / 2) u^l \quad \text{and} \quad u^{l+1} = \widetilde{S}(\tau / 2) \mathcal{S}(\tau / 2) u^l = \widetilde{\mathcal{S}}(\tau / 2) v^l.
		\end{aligned}
	\end{equation*}
	Applying Theorem \ref{mbp_ess1} and \ref{mbp_ess1_adjoint}, we get
	\begin{equation*}
		\begin{aligned}
			\|v^l\|_\infty & = \|\mathcal{S}(\tau / 2) u^l\|_\infty \leq \beta \quad \text{and} \quad
			\|u^{l+1}\|_\infty = \|\widetilde{S}(\tau / 2) v^l\|_\infty \leq \beta.
		\end{aligned}
	\end{equation*}
	The result is also true $n=l+1$, and thus we have completed the proof of Theorem \ref{dmp_ss2}.
\end{proof}
It is worth mentioning that the DMP of the \textbf{SS2-adjoint} can be obtained through a comparable approach. Thus, we only list the result and omit the detailed proof. 
\begin{thm}[DMP of \textbf{SS2-adjoint} method] \label{dmp_ss2_adj}
	Assume that the initial value $u^0$ is bounded in the sense $\|u^0\|_\infty \leq \beta$. Then the solution generated by \textbf{SS2-adjoint} satisfies
	\begin{equation*}
		\|u^n\|_{\infty} \leq \beta, \quad  \ 0 \leq n \leq N_t.
	\end{equation*}
	Provided that the time step $0 < \tau \leq \min\{ \frac{2h^2}{\kappa h^2 + 2\varepsilon^2}, \frac{2}{\kappa + \|f^\prime\|_{C[-\beta, \beta]}}  \}$.
\end{thm}
\begin{rmk}
	In the above theorems, the restrictions on the time step imply that $\tau = \mathcal{O}(h^2 / \varepsilon^2)$, which are similar to those imposed in \cite{hou_cnab,hou_leapfrog,ju_esav}. These constraints mainly arises from the explicit discretization of  $\varDelta_h u$. We can also observe that the step size limitations are relaxed in the second-order schemes.
\end{rmk}
\subsection{DEDL of the proposed methods}
The following Lemma \ref{dg_property} is indispensable in establishing the energy stability of the proposed methods.
\begin{lem}\label{dg_property}
	Let $v, w \in \mathbb{M}_h$. Then, we have
	\begin{equation*}
		-\langle \varDelta_a v + \varDelta_b w, w - v \rangle = -\langle \varDelta_a w + \varDelta_b v, w - v \rangle = \frac{1}{2}\|\nabla_h w\|^2 - \frac{1}{2}\|\nabla_h v\|^2.
	\end{equation*}
\end{lem}
	\begin{proof}
		We can easily obtain the following decompositions due to the fact that $\varDelta_h = \varDelta_a + \varDelta_b$
		\begin{equation*}
			\begin{aligned}
				 & \varDelta_a v + \varDelta_b w = \frac{1}{2}\varDelta_h (w + v) + \frac{1}{2} (\varDelta_b - \varDelta_a) (w - v), \\
				 & \varDelta_a v + \varDelta_b w = \varDelta_a w + \varDelta_b v + (\varDelta_b - \varDelta_a) (w - v).
			\end{aligned}
		\end{equation*}
		By applying the skew-adjoint of $\varDelta_b - \varDelta_a$ and Proposition \ref{summation_by_parts}, we arrive at the result of Lemma \ref{dg_property}.
	\end{proof}

\begin{lem}[\cite{csrk}]\label{convex_splitting}
	Suppose that $v$ and $w$ are sufficiently smooth and $F(\cdot)$ is twice continuously differentiable. Consider the canonical convex splitting of $F(\cdot)$ into $F(\cdot) = F_c(\cdot) - F_e(\cdot)$, where both $F_c(\cdot)$ and $F_e(\cdot)$ are convex. Then,
	\begin{equation*}
		F(\xi) - F(\eta) \leq (F^\prime_c (\xi) - F^\prime_e (\eta)) \cdot (\xi - \eta).
	\end{equation*}
\end{lem}
\begin{coro}\label{corol}
    By setting $\kappa$ large enough such that $\kappa \geq \|f^\prime\|_{C[-\beta, \beta]}$, 
we can construct a convex splitting of bulk energy as
\begin{equation*} 
    F(\xi) = F_c(\xi) - F_e(\xi) = \frac{\kappa}{2} \xi^2 - (\frac{\kappa}{2} \xi^2 - F(\xi)).
\end{equation*}
It follows from Lemma.~\ref{convex_splitting} that
\begin{equation*}
	\langle F(u^{n+1}) - F(u^n), 1 \rangle \leq \langle \kappa u^{n+1} - (\kappa u^n + f(u^n)), u^{n+1} - u^n \rangle.
\end{equation*}
\end{coro}

\begin{thm}[DEDL of \textbf{ESS1} method]\label{energy_ess1}
				Under the assumptions of Theorem \ref{mbp_ess1}, the solution generated by the \textbf{ESS1} method satisfies
	\begin{equation}
		E_h(u^{n+1}) \leq E_h(u^n), \quad 0 \leq n \leq N_t - 1.
	\end{equation}
\end{thm}
\begin{proof}
	Taking the discrete inner product on both sides of \eqref{ess1} by $-\tau \delta_t u^{n+\frac{1}{2}}$ and using Lemma \ref{dg_property}, Corollary \ref{corol}, it is straightforward to get the desired result.
\end{proof}

\begin{thm}[DEDL of \textbf{ESS1-adjoint} method]\label{energy_ess1_adjoint}
	Under the assumption of Theorem \ref{mbp_ess1_adjoint}, the solution generated by the \textbf{ESS1-adjoint} method satisfies
	\begin{equation*}
		E_h (u^{n+1}) \leq E_h(u^n), \quad  0 \leq n \leq N_t - 1.
	\end{equation*}
\end{thm}

\begin{proof}
	Taking the discrete inner product on both sides of \eqref{ess1_adjoint} by $\tau \delta_t u^{n+\frac{1}{2}}$, using Lemma \ref{dg_property} and making some arrangements to obtain
	\begin{equation*}
		(\kappa -\frac{1}{\tau}) \|u^{n+1} - u^n\|^2 = \frac{\varepsilon^2}{2} \|\nabla_h u^{n+1}\|^2 - \frac{\varepsilon^2}{2} \|\nabla_h u^n\|^2 - \langle f(u^{n+1}), \tau \delta_t u^{n+\frac{1}{2}} \rangle,
	\end{equation*}
	that is
	\begin{equation}\label{energy_ess1_adjoint_eq1}
		E_h(u^{n+1}) - E_h(u^n)
		= -(\frac{1}{\tau} - \kappa)\|u^{n+1} - u^n\|^2 + \langle F(u^{n+1}) - F(u^n), 1 \rangle + \langle f(u^{n+1}), u^{n+1} - u^n \rangle. \\
	\end{equation}
	The mean-value theorem and the fact $F^\prime = -f$ then yield
	\begin{equation*}
		\begin{aligned}
			 & \langle F(u^{n+1}) - F(u^n), 1 \rangle + \langle f(u^{n+1}), u^{n+1} - u^n \rangle                                                                                        \\
			 & \quad = h^2 \sum\limits_{i=0}^{M-1}\sum\limits_{j=0}^{M-1} \left( F(u_{ij}^{n+1}) - F(u_{ij}^n) + f(u_{ij}^{n+1})(u_{ij}^{n+1} - u_{ij}^{n}) \right)                  \\
			 & \quad = h^2 \sum\limits_{i=0}^{M-1}\sum\limits_{j=0}^{M-1} \left( f(u_{ij}^{n+1}) - f(\xi_{ij}^{n+1}) \right) (u_{ij}^{n+1} - u_{ij}^n)                                \\
			 &\quad \leq h^2 \|f^\prime\|_{C[-\beta, \beta]} \sum\limits_{i=0}^{M-1} \sum\limits_{j=0}^{M-1}  |u_{ij}^{n+1} - \xi_{ij}^{n+1}| \cdot |u_{ij}^{n+1} - u_{ij}^n|.
		\end{aligned}
	\end{equation*}
	It follows from the fact $\xi_{ij}^{n+1} \in (\min \{u_{ij}^n, u_{ij}^{n+1}\}, \max \{u_{ij}^n, u_{ij}^{n+1}\})$ that
	\begin{equation}\label{energy_ess1_adjoint_eq2}
		\langle F(u^{n+1}) - F(u^n), 1 \rangle + \langle f(u^{n+1}), u^{n+1} - u^n \rangle \leq \|f^\prime\|_{C[-\beta, \beta]} \|u^{n+1} - u^n\|^2. \\
	\end{equation}
	We obtain by substituting \eqref{energy_ess1_adjoint_eq2} into \eqref{energy_ess1_adjoint_eq1} that
	\begin{equation*}
		E_h(u^{n+1}) - E_h(u^n) \leq -(\frac{1}{\tau} -\kappa - \|f^\prime\|_{C[-\beta, \beta]})\|u^{n+1} - u^n\|^2 \leq 0,
	\end{equation*}
	which provides us the result of Theorem \ref{energy_ess1_adjoint}.
\end{proof}
\begin{thm}[Energy dissipation of \textbf{SS2} method]
				Under the assumption of Theorem \ref{dmp_ss2}, the solution generated by the \textbf{SS2} method satisfies
	\begin{equation*}
		E_h (u^{n+1}) \leq E_h(u^n), \quad 0 \leq n \leq N_t- 1.
	\end{equation*}
\end{thm}
\begin{proof}
	Let
	\begin{equation*}
		\begin{aligned}
			v^n = \mathcal{S}(\tau / 2) u^n \quad \text{and} \quad u^{n+1} = \widetilde{S}(\tau / 2) \mathcal{S}(\tau / 2) u^n = \widetilde{\mathcal{S}}(\tau / 2) v^n.
		\end{aligned}
	\end{equation*}
	Theorems \ref{energy_ess1}, \ref{energy_ess1_adjoint} give us
	\begin{equation*}
		E_h(v^n) \leq E_h(u^n), \quad E_h(u^{n+1}) \leq E_h(v^n).
	\end{equation*}
	Adding the above equations together leads to the desired result.
\end{proof}
\begin{thm}[Energy dissipation of \textbf{SS2-adjoint} method]
				Under the assumption of Theorem \ref{dmp_ss2_adj}, the solution generated by the \textbf{SS2-adjoint} method satisfies
	\begin{equation*}
		E_h (u^{n+1}) \leq E_h(u^n), \quad  0 \leq n \leq N_t - 1.
	\end{equation*}
\end{thm}
It is worth mentioning that the DEDL of the \textbf{SS2-adjoint} can be obtained through a comparable approach. Thus, we omit the detailed proof. 

\section{Convergence analysis}\label{convergence}
Using the energy method, we present convergence analysis of the proposed methods with respect to the norm defined as
\begin{equation*}
	\|v\|^2_\kappa = \frac{\varepsilon^2}{2} \|\nabla_h v\|^2 + \kappa \|v\|^2, \quad v \in \mathbb{M}_h.
\end{equation*}
To simplify the derivations, we denote by $L = \|f^\prime\|_{C[-\beta, \beta]}$ and by $C$ a generic positive constant independent of the discretization parameters. In this paper, we provide rigorous analysis of the convergence of the \textbf{ESS1} and \textbf{SS2} methods, while we omit the analysis of the \textbf{ESS1-adjoint} and \textbf{SS2-adjoint} methods as they are similar.
\begin{thm}\label{conv_ess1}
	Suppose that the exact solution of \eqref{ac} $u_e(x, y, t)\in C^2([0 ,T]; C^4(\overline{\Omega}))$, and the initial value $\|u_e(0)\|_\infty \leq \beta$. Then the $\textbf{ESS1}$ scheme is convergent in the sense
	\begin{equation}\label{convergence_ineq_ess1}
					\|u_e(t_n) - u^n\|_\kappa \leq C (\frac{\varepsilon^2 \tau}{h} + \tau + \varepsilon^2 h^2 ), \quad \forall \ 0 \leq n \leq \left[\frac{T}{\tau}\right].
	\end{equation}
	Provided that the time step $0 < \tau \leq \frac{h^2}{2\varepsilon^2}$. Consequently, if the spatial step satisfies $h \geq C \varepsilon^2$, then the \textbf{ESS1} scheme is convergent in the sense
	\begin{equation*}
					\|u_e(t_n) - u^n\|_\kappa \leq C(\tau + \varepsilon^2 h^2).
	\end{equation*}
\end{thm}
\begin{proof}
	The exact solution can be regarded as satisfying \eqref{ess1} with a defect $\mathcal{R}_1$, such that
	\begin{equation}\label{exact_solution_defect}
		\begin{aligned}
			\frac{u_e(t_{n+1}) - u_e(t_n)}{\tau} & = \varepsilon^2 (\varDelta_a u_e(t_n) + \varDelta_b u_e(t_{n+1})) + f(u_e(t_n)) \\
			                                     & - \kappa (u_e(t_{n+1}) - u_e(t_n)) + \mathcal{R}_1.
		\end{aligned}
	\end{equation}
	Using the Taylor's formula, we obtain
	\begin{equation*}
		\|\mathcal{R}_1\| \leq C(\frac{\varepsilon^2 \tau}{h} + \tau + \varepsilon^2 h^2).
	\end{equation*}
	Let $e^n = u_e(t_n) - u^n$ be the solution error, which satisfies the error equation obtained by subtracting \eqref{ess1} from \eqref{exact_solution_defect}
	\begin{equation}\label{erreq_ess1_convergence}
		\delta_t e^{n+\frac{1}{2}} = \varepsilon^2 (\varDelta_a e^n + \varDelta_b e^{n+1} ) + f(u_e(t_n)) - f(u^n) - \kappa (e^{n+1} - e^n) + \mathcal{R}_1.
	\end{equation}
	Taking the discrete inner product on both sides of \eqref{erreq_ess1_convergence} by $\tau \delta_t e^{n+\frac{1}{2}}$, using Lemma \ref{dg_property}, and rearranging the resulting equation, we get
	\begin{equation*}
		\begin{aligned}
			& \frac{1}{\tau}\|e^{n+1} - e^n\|^2 + \frac{\varepsilon^2}{2} \|\nabla_h e^{n+1}\|^2 - \frac{\varepsilon^2}{2}\|\nabla_h e^n\|^2 + \kappa \|e^{n+1} - e^n\|^2 \\
			& \quad \leq \langle f(u_e(t_n)) - f(u^n), e^{n+1} - e^n \rangle + \langle \mathcal{R}_1, e^{n+1} - e^n \rangle.
		\end{aligned}
	\end{equation*}
	Applying the identity $(a - b)^2 = a^2 - b^2 - 2b(a-b)$, Cauch-Schwarz and Young's inequalities, and Theorem \ref{mbp_ess1}, we can derive the following estimate:
	\begin{equation}\label{derivation_gronwall}
		\begin{aligned}
			& \frac{1}{\tau}\|e^{n+1} - e^n\|^2 + \|e^{n+1}\|_\kappa^2 - \|e^n\|_\kappa^2                                                           \\
			& \quad \leq  L \|e^{n+1}-e^n\|^2 + 2\kappa\|e^n\| \|e^{n+1} - e^n\| + \|\mathcal{R}_1\|\|e^{n+1} - e^n\|                                 \\
			& \quad \leq  (2\kappa + L)(\|e^n\| + \|e^{n+1}\| + \|\mathcal{R}_1\|)  \|e^{n+1} - e^n\|                                    \\
			& \quad \leq \frac{\tau}{4} \big( (2\kappa + L)(\|e^n\| + \|e^{n+1}\| + \|\mathcal{R}_1\| \big)^2 + \frac{1}{\tau}\|e^{n+1} - e^n\|^2,
		\end{aligned}
	\end{equation}
	that is 
	\begin{equation*}
		\|e^{n+1}\|_\kappa^2 - \|e^n\|_\kappa^2 \leq \frac{(2\kappa + L)^2}{2} \tau (\|e^{n+1}\|_\kappa^2 + \|e^n\|_\kappa^2) + C\tau \|\mathcal{R}_1\|^2.
	\end{equation*}
	Therefore, with the time step restriction in place, we can use the discrete Gronwall inequality to obtain
	\begin{equation*}
		\|e^n\|^2_\kappa \leq \exp{((2\kappa + L)^2 n \tau)} \|e^0\|^2_\kappa + C (\frac{\varepsilon^2 \tau}{h} + \tau + \varepsilon^2 h^2)^2.
	\end{equation*}
	Upon incorporating the initial condition $\|e^0\|_\kappa = 0$, we arrive at the desired inequality \eqref{convergence_ineq_ess1}, and hence complete the proof.
\end{proof}
\begin{thm}\label{convergence_thm_ss2}
	Suppose that the exact solution of \eqref{ac} satisfies $u_e(x, y, t) \in C^3([0,T]; C^4(\overline{\Omega}))$ and the initial value $\|u_e(0)\| \leq \beta$. Then the \textbf{SS2} scheme is convergent in the sense
	\begin{equation}\label{convergence_ineq_ss2}
		\| u_e(t_n) - u^n\|_\kappa \leq C(\frac{\varepsilon^4 \tau^2}{h^3} + \frac{\varepsilon^2 \tau^2}{h} + \tau^2 + \varepsilon^2 h^2),  \quad \forall \ 0 \leq n \leq \left[\frac{T}{\tau}\right].
	\end{equation}
	Provided that $\kappa \geq L$ and $0 < \tau \leq \min\{ \frac{2}{\kappa + L}, \frac{2h^2}{\kappa^2 + 2\varepsilon^2}, \frac{\kappa^2}{(L+4\kappa)^2} \}$. Consequently, if the spatial step satisfies $h \geq C \varepsilon^{\frac{4}{3}}$, then the \textbf{SS2} scheme is convergent in the sense
	\begin{equation*}
					\|u_e(t_n) - u^n\|_\kappa \leq C( \tau^2 + \varepsilon^2 h^2 ).
	\end{equation*}
\end{thm}

	We introduce a reference solution $v^n \in \mathbb{M}_h$, such that
	\begin{equation}\label{ss2_ref}
		\begin{aligned}
			\frac{ v^n - u_e(t_n) }{\tau/2}   & = \varepsilon^2 (\varDelta_a u_e(t_n) + \varDelta_b v^n) + f(u_e(t_n)) - \kappa (v^n - u_e(t_n)),                            \\
			\frac{u_e(t_{n+1}) - v^n}{\tau/2} & = \varepsilon^2 (\varDelta_a u_e(t_{n+1}) + \varDelta_b v^n) + f(u_e(t_{n+1})) + \kappa(u_e(t_{n+1}) - v^n) + \mathcal{R}_2.
		\end{aligned}
	\end{equation}
	The following lemma provides the estimate of $\mathcal{R}_2$.
	\begin{lem}
		Suppose that the exact solution of \eqref{ac} satisfies $u_e(x, y, t) \in C^3([0,T]; C^4(\overline{\Omega}))$, then, 
		\begin{equation*}
			\|\mathcal{R}_2\| \leq C( \frac{\varepsilon^4\tau^2}{h^3} + \frac{\varepsilon^2\tau^2}{h} + \tau^2 + \varepsilon^2 h^2 ).
		\end{equation*}
	\end{lem}
\begin{proof}
		To simplify the presentation, we set $\kappa = 0$ and define
		\begin{equation*}
			\mathscr{L} := \varepsilon^2 \Delta \quad  \text{and} \quad \mathscr{L}_\alpha := \varepsilon^2 \Delta_\alpha, \ \alpha = a,b,h, \quad 
		\end{equation*}
		Adding equations in \eqref{ss2_ref} together yields 
		\begin{equation*}
			\frac{u_e(t_{n+1}) - u_e(t_n)}{\tau} = \mathscr{L}_a \frac{u_e(t_n) + u_e(t_{n+1})}{2} + \mathscr{L}_b v^n + \frac{f(u_e(t_n)) + f(u_e(t_{n+1}))}{2} + \frac{1}{2}\mathcal{R}_2.
		\end{equation*}
		Comparing it with \eqref{ac}, we have
		\begin{equation*}
			\frac{1}{2} \mathcal{R}_2 = \mathcal{T}_1 + \mathcal{T}_2 + \mathcal{T}_3 + \mathcal{T}_4 + \mathcal{T}_5,
		\end{equation*}
		where
		\begin{equation*}
			\begin{aligned}
				(\mathcal{T}_1)_{ij} &= \frac{ u_e(x_i, y_j, t_{n+1}) - u_e(x_i, y_j, t_n) }{\tau} - \partial_t u_e(x_i, y_j, t_{n+\frac{1}{2}}),  \\ 
				(\mathcal{T}_2)_{ij} &= \mathscr{L} u_e (x_i, y_j, t_{n+\frac{1}{2}}) - \mathscr{L}_h u_e(x_i, y_j, t_{n+\frac{1}{2}}), \\ 
				(\mathcal{T}_3)_{ij} &= f(u_e(x_i, y_j, t_{n+\frac{1}{2}})) - \frac{f(u_e(x_i, y_j, t_n)) + f(u_e(x_i, y_j, t_{n+1}))}{2}, \\
				(\mathcal{T}_4)_{ij} &= \mathscr{L}_a u_e(x_i, y_j, t_{n+\frac{1}{2}}) - \mathscr{L}_a \frac{u_e(x_i, y_j, t_n) + u_e(x_i, y_j, t_{n+1})}{2}, \\
				(\mathcal{T}_5)_{ij} &= \mathscr{L}_b u_e(x_i, y_j, t_{n+\frac{1}{2}}) - \mathscr{L}_b v_{ij}^n. \\ 
			\end{aligned}
		\end{equation*}
		An application of the Taylor's formula yields
		\begin{equation*}
			(\mathcal{T}_1)_{ij} \leq \frac{\tau^2}{24} \| u_e \|_{C^3([0, T]; C(\overline{\Omega}))}, \
			(\mathcal{T}_2)_{ij} \leq \frac{\varepsilon^2h^2}{6} \|u_e \|_{C([0, T], C^4(\overline{\Omega})) }, \
			(\mathcal{T}_3)_{ij}  \leq \tau^2 \|f\|_{C^2([-1, 1])} \|u_e\|_{C^2([0,T]; C(\overline{\Omega}))}.
		\end{equation*}
		\begin{equation*}
			(\mathcal{T}_4)_{ij}  = 
			\left\lbrace
			\begin{aligned}
				&\mathcal{O}(\varepsilon^2\tau^2/h), \quad && 0 < i,j < M-1, \\ 
				&\mathcal{O}(\varepsilon^2\tau^2/h^2), \quad && i,j=0, M-1.
			\end{aligned}
			\right.
		\end{equation*}
		To estimate $\mathcal{T}_5$, we introduce
		\begin{equation*}
			\begin{aligned}
				&\mathscr{P}_h = ( I - \frac{\tau}{2}\mathscr{L}_b )^{-1} ( I + \frac{\tau}{2}\mathscr{L}_a ) - I = \frac{\tau}{2} \mathscr{L}_h + \sum\limits_{k=2}^{\infty} \frac{\tau^k}{2^k} \mathscr{L}_b^{k-1}\mathscr{L}_h, \\ 
				&\mathscr{Q}_h = \frac{\tau}{2} ( I - \frac{\tau}{2}\mathscr{L}_b )^{-1} - \frac{\tau}{2} I = \sum\limits_{k=2}^\infty \frac{\tau^k}{2^k} \mathscr{L}_b^{k-1}.
			\end{aligned}
		\end{equation*}
		Since $\tau < 2h^2/\varepsilon^2$, it is readily to confirm that the spectral radius $\rho(\tau \mathscr{L}_b/2)<1$, and the above expansions make sense. The first equation of \eqref{ss2_ref} then becomes
		\begin{equation}\label{solut_v}
			v^n = \mathscr{P}_h u_e(t_n) + \mathscr{Q}_h f(u_e(t_n)) + u_e(t_n) + \frac{\tau}{2} f(u_e(t_n)).
		\end{equation}
		It holds by combining \eqref{solut_v} with the fact
		\begin{equation*}
			\begin{aligned}
				u_e(t_{n+\frac{1}{2}}) &= u_e(t_n) + \frac{\tau}{2} \partial_t u_e(t_n) + \int_{t_n}^{t_{n+\frac{1}{2}}} (s - t_n) \partial_{tt} u_e(s) ds \\ 
				&= u_e(t_n) + \frac{\tau}{2} \mathscr{L} u_e(t_n) + \frac{\tau}{2} f(u_e(t_n)) + \int_{t_n}^{t_{n+\frac{1}{2}}} (s - t_n) \partial_{tt} u_e(s) ds,
			\end{aligned}
		\end{equation*}
		that
		\begin{equation*}
			\begin{aligned}
				\mathcal{T}_5 &= \int_{t_n}^{t_{n+\frac{1}{2}}} (s-t_n) \mathscr{L}_b \partial_{tt} u_e(s) ds - \sum\limits_{k=2}^\infty (\frac{\tau \mathscr{L}_b}{2})^k  \partial_t u_e(t_n) \\ 
				&= \int_{t_n}^{t_{n+\frac{1}{2}}} (s-t_n) \mathscr{L}_b \partial_{tt} u_e(s) ds - (I - \frac{\tau \mathscr{L}_b}{2})^{-1} (\frac{\tau \mathscr{L}_b}{2})^2 \partial_t u_e(t_n).
			\end{aligned}
		\end{equation*}
		There is no difficulty in confirming that for any $v \in \mathbb{M}_h$, 
		\begin{equation*}
			\| (I - \frac{\tau \mathscr{L}_b}{2})^{-1} v \|_\infty \leq 2 \|v\|_\infty.
		\end{equation*}
		Furthermore, by the Taylor's formula and some straightforward calculations, we have
		\begin{equation*}
			(\mathcal{T}_5)_{ij} =
			\left\lbrace
			\begin{aligned}
				&\mathcal{O}(\varepsilon^2\tau^2/h + \varepsilon^4\tau^2/h^2), && \quad 1 < i,j < M-1, \\
				&\mathcal{O}(\varepsilon^2\tau^2/h^2 + \varepsilon^4 \tau^2/h^4), && \quad i,j = 0,1,M-1.
			\end{aligned}
			\right.
		\end{equation*}
		Combining all estimates, we obtain 
		\begin{equation*}
			\|\mathcal{R}_2\| \leq C( \frac{\varepsilon^4\tau^2}{h^3} + \frac{\varepsilon^2\tau^2}{h} + \tau^2 + \varepsilon^2 h^2 ),
		\end{equation*}
		which completes the proof.
\end{proof}
\begin{rmk}
    We remark here that the presence of $\mathcal{O}( \frac{\varepsilon^4 \tau^2}{h^2} )$ term in the truncation error is primarily introduced by the boundary points. This is due to the modifications made at the boundary in the extended Saul'yev methods.
\end{rmk}
\begin{proof}
	By the proof of the Theorem.~\ref{mbp_ess1}, it is straightforward to confirm that the intermediate solution remain satisfies $\|v^n\|_\infty \leq \beta$.  We define the solution errors as
	\begin{equation*}
		e^n = u_e(t_n)  - u^n, \quad \psi^n = v^n - \mathcal{S}(\tau/2) u^n,
	\end{equation*}
	which satisfy the following error equations, obtained by making the difference between \eqref{ss2_ref} and the \textbf{SS2} method
	\begin{equation}\label{error_ss2}
		\begin{aligned}
			\frac{\psi^n - e^n}{\tau / 2}   & = \varepsilon^2 (\varDelta_a e^n + \varDelta_b \psi^n) + f(u_e(t_n)) - f(u^n) - \kappa(\psi^n - e^n),                                  \\
			\frac{e^{n+1} - \psi^n}{\tau/2} & = \varepsilon^2 (\varDelta_a e^{n+1} + \varDelta_b \psi^n) + f(u_e(t_{n+1})) - f(u^{n+1}) + \kappa (e^{n+1} - \psi^n) + \mathcal{R}_2.
		\end{aligned}
	\end{equation}
	The first equation in \eqref{error_ss2} and the proof of Theorem.~\ref{conv_ess1} give us
	\begin{equation*}
		\|\psi^n\|^2_\kappa - \|e^n\|^2_\kappa \leq \frac{(2\kappa + L)^2}{4} \tau (\|\psi^n\|_\kappa^2 + \|e^n\|_\kappa^2).
	\end{equation*}
	Under the given step size restriction, it holds that
	\begin{equation}\label{ss2_gronwall1}
		\|\psi^n\|^2_\kappa \leq (1 + \frac{(2\kappa + L)^2}{2}\tau)\|e^n\|_\kappa^2.
	\end{equation}
	Rearranging the second equation of \eqref{error_ss2}, we obtain
	\begin{equation}\label{error_ss2_rearrange}
		\begin{aligned}
			 & \frac{e^{n+1} - \psi^n}{\tau/2} - \varepsilon^2 (\varDelta_a e^{n+1} + \varDelta_b \psi^n) + \kappa (e^{n+1} - \psi^n) \\
			 & \quad = f(u_e(t_{n+1})) - f(u^{n+1}) + 2\kappa (e^{n+1} - \psi^n) + \mathcal{R}_2.
		\end{aligned}
	\end{equation}
	Taking the discrete inner product on both sides of \eqref{error_ss2_rearrange}  by $e^{n+1} - \psi^n$, using the property $(a - b)^2 = a^2 - b^2 + 2b(b-a)$, we have
	\begin{equation*}
		\begin{aligned}
			 & \frac{2}{\tau}\|e^{n+1} - \psi^n\|^2 + \|e^{n+1}\|_\kappa^2 - \|\psi^n\|_\kappa^2 = \langle f(u_e(t_{n+1})) - f(u^{n+1}), e^{n+1} - \psi^n \rangle   \\
			 & \hspace{1cm} + 2\kappa\| e^{n+1} - \psi^n\|^2 + 2\kappa\langle \psi^n,  e^{n+1} - \psi^n \rangle + \langle \mathcal{R}_2, e^{n+1} - \psi^n \rangle .
		\end{aligned}
	\end{equation*}
	Using the Cauchy-Schwarz and Young's inequalities to further simplify the above equation, it holds that
	\begin{equation*}
		\begin{aligned}
			 & \frac{2}{\tau}\|e^{n+1} - \psi^n \|^2 + \|e^{n+1}\|_\kappa^2 - \| \psi^n \|_\kappa^2                                                                \\
			 & \quad \leq \big( (L + 2\kappa) \|e^{n+1}\| + 4\kappa\|\psi^n\| + \|\mathcal{R}_2\| \big) \|e^{n+1} - \psi^n\|                                       \\
			 & \quad  \leq \frac{\tau}{2} \big( (L + 2\kappa) \|e^{n+1}\| + 4\kappa\|\psi^n\| + \|\mathcal{R}_2\| \big)^2 + \frac{2}{\tau} \|e^{n+1} - \psi^n\|^2  \\
			 & \quad  \leq 2 (L + 4\kappa)^2 \tau (\|e^{n+1}\|_\kappa^2 + \|\psi^n\|_\kappa^2) + \tau \|\mathcal{R}_2\|^2 + \frac{2}{\tau} \|e^{n+1} - \psi^n\|^2,
		\end{aligned}
	\end{equation*}
	that is,
	\begin{equation*}
		\|e^{n+1}\|_\kappa^2 - \|\psi^n\|_\kappa^2  \leq  2(L + 4\kappa)^2\tau (\|e^{n+1}\|_\kappa^2 + \|\psi^n\|_\kappa^2) + \tau \|\mathcal{R}_2\|^2.
	\end{equation*}
	The provided time step restriction then leads to
	\begin{equation}\label{ss2_gronwall2}
		\|e^{n+1}\|_\kappa^2 \leq (1 + 4(L + 4\kappa)^2\tau )\|\psi^n\|_\kappa^2 + 2\tau \|\mathcal{R}_2\|^2.
	\end{equation}
	It holds by inserting \eqref{ss2_gronwall1} into \eqref{ss2_gronwall2} and dropping some useless terms that
	\begin{equation}\label{grownwall}
					\|e^{n+1}\|_\kappa^2  \leq ( 1 + C\tau )\|e^n\|_\kappa^2  + C \tau \|\mathcal{R}_2\|^2.
	\end{equation}
	Solving the recursive formula \eqref{grownwall}, we finally get
	\begin{equation*}
		\begin{aligned}
						\|e^n\|^2_\kappa & \leq \exp{(1 + C n \tau)} \|e^0\|^2_\kappa + C \tau \sum\limits_{j=0}^{n-1} (1 + C\tau)^j \|\mathcal{R}_2\|^2 \\
			                 & \quad \leq C (\frac{\varepsilon^4 \tau^2}{h^3} + \frac{\varepsilon^2 \tau^2}{h} + \tau^2 + \varepsilon^2 h^2)^2.   \\
		\end{aligned}
	\end{equation*}
	Consequently, the convergent result \eqref{convergence_ineq_ss2} holds. The proof is thus completed.
\end{proof}

\section{Numerical experiments}\label{numerical_experiments}
In this section, we demonstrate numerical performance of the proposed methods in terms of computational efficiency, DEDL, and DMP for simulating the AC equation \eqref{ac} with nonlinear terms given by \eqref{f_double_well} and \eqref{f_log}. For the nonlinear term induced by the double-well potential \eqref{f_double_well}, we have $\beta = 1$ and set the stabilized parameter $\kappa = \|f^\prime\|_{C[-\beta, \beta]} = 2$. For the nonlinear term \eqref{f_log}, we fix $\theta = 0.8$ and $\theta_c = 1.6$. Then $\beta \approx 0.9575$ will be the positive root of $f^\prime(\xi) = 0$, and the stabilized parameter is thus specified as $\kappa = \|f^\prime\|_{C[-\beta, \beta]} \approx 8.02$. The nonlinear system of \eqref{ess1_adjoint_form2} is solved using the Newton's method with a tolerance threshold set to $10^{-12}$, which is sufficient for mediocre accuracy requirements.

Let $\Omega = (0, 1)^2$ and $\varepsilon = 0.01$, we begin with verifying the convergence rate of proposed methods against both spatial and time steps by specifying the initial condition $u_0(x, y) = 0.1 \sin{(2\pi x)} \sin{(2\pi y)}.$ Since the exact solution is unavailable, we use the solution produced by \textbf{SS2} method with step sizes $h = 1/1024$ and $\tau = 10^{-6}$ as a reference to better calculate errors in the refinement tests. We systematically adjust $h$, $\tau$ and integrate the AC equation until $t = 1$ using various methods. The numerical and reference solutions at $t=1$ are compared, and their errors in the discrete $L^2$ norm are computed.
\begin{figure}[H]
	\centering
	\subfigure{\label{sinesine_doublewell_time_accuracy}
		\includegraphics[width = 0.45\linewidth]{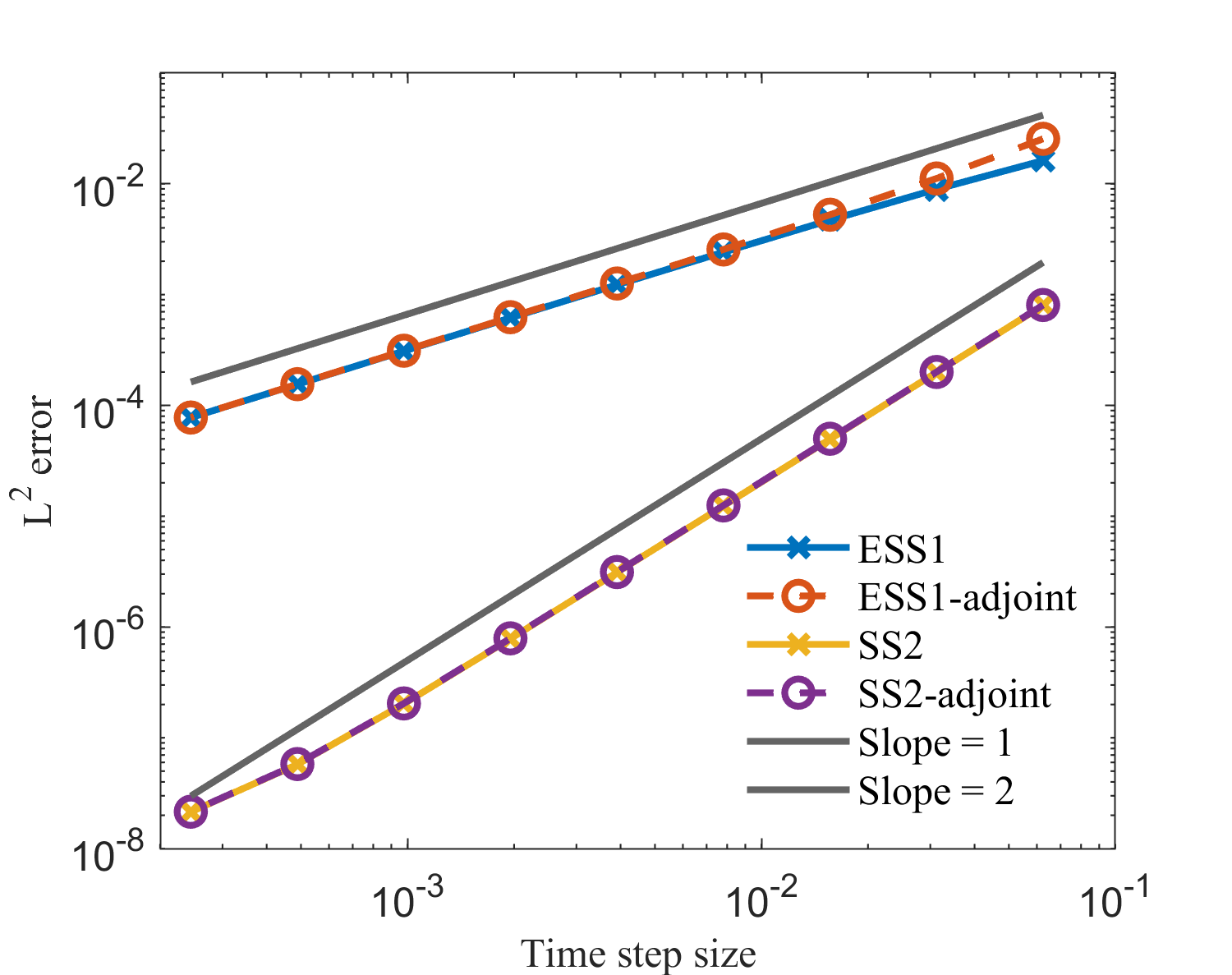}
	}
	\subfigure{\label{sinesine_log_time_accuracy}
		\includegraphics[width = 0.45\linewidth]{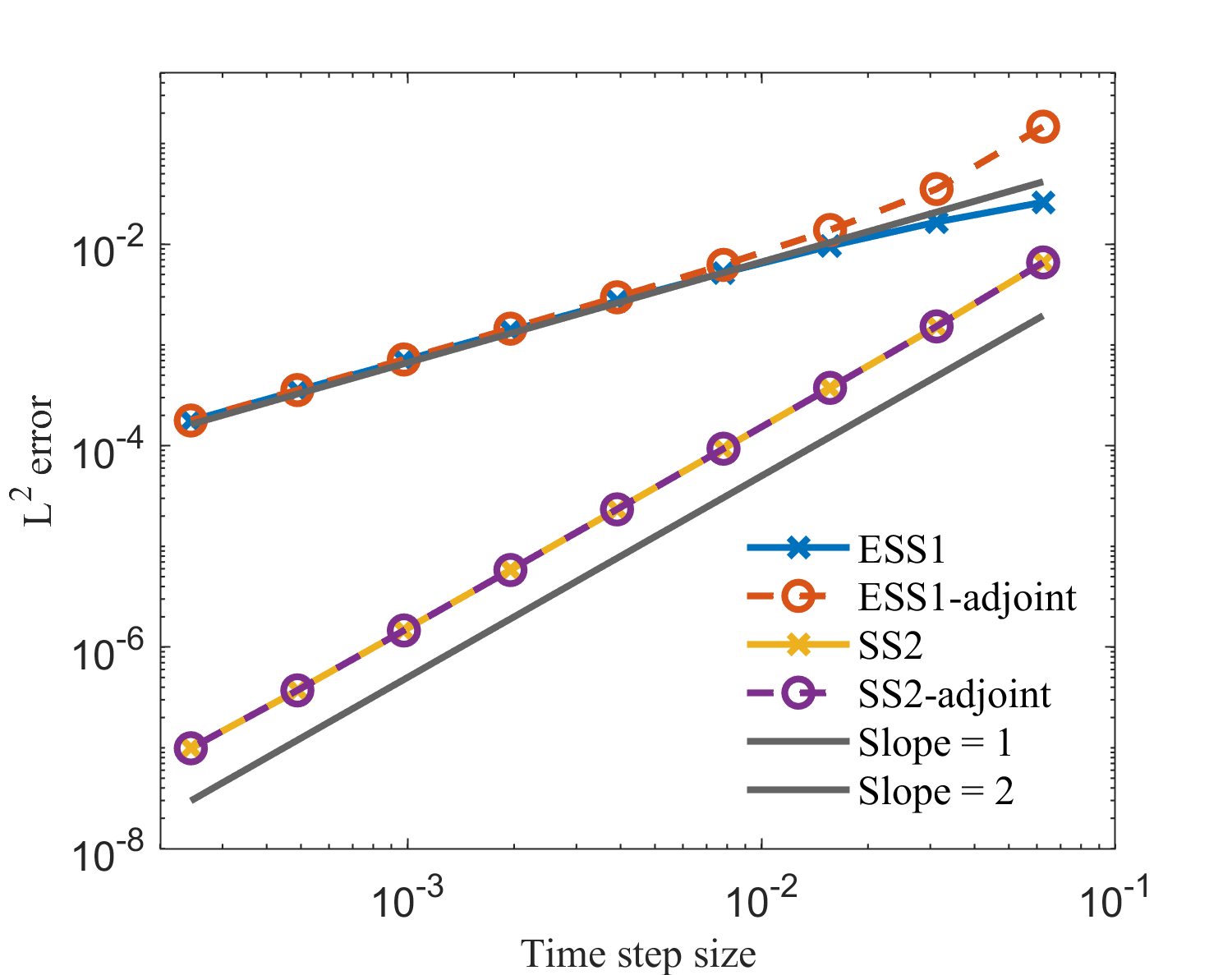}
	}
	\caption{ \label{sinesine_time_accuracy} The $L^2$ error vs the time step size computed by different methods for the double-well potential (left) and the Flory-Huggins potential (right).}
\end{figure}
For the temporal accuracy test, we fix the integration time at $t=1$ and set $h = 1/512$ such that the error caused by the spatial discretization is negligible. To execute the refinement tests in time, we vary $\tau = 2^{-k}$ with $k$ ranging from $4$ to $12$. The convergence results for solving the AC equation with double-well and the Flory-Huggins potential are respectively reported in the left and right subplots of Figure \ref{sinesine_time_accuracy}. The first- and second-order accuracy of $\textbf{ESS1}$ and \textbf{SS2} methods, as well as their adjoint methods, are observed as expected in both circumstances. We also discover that the truncation error is marginally larger while solving \eqref{ac} with Flory-Huggins potential due to a larger choice of the stabilized parameter $\kappa$.

\begin{figure}[H]
	\centering
	\includegraphics[width = 0.45\linewidth]{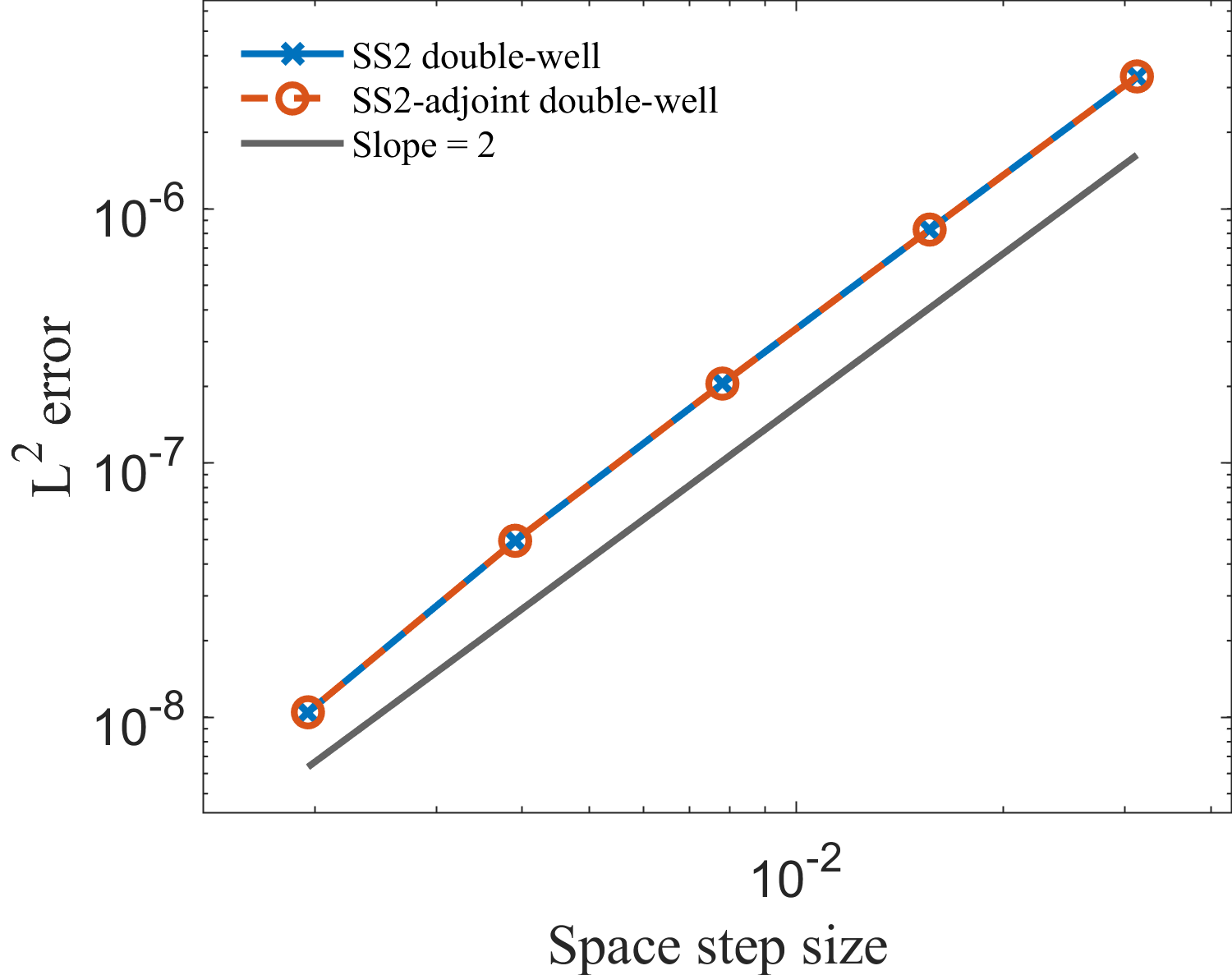}
	\includegraphics[width = 0.45\linewidth]{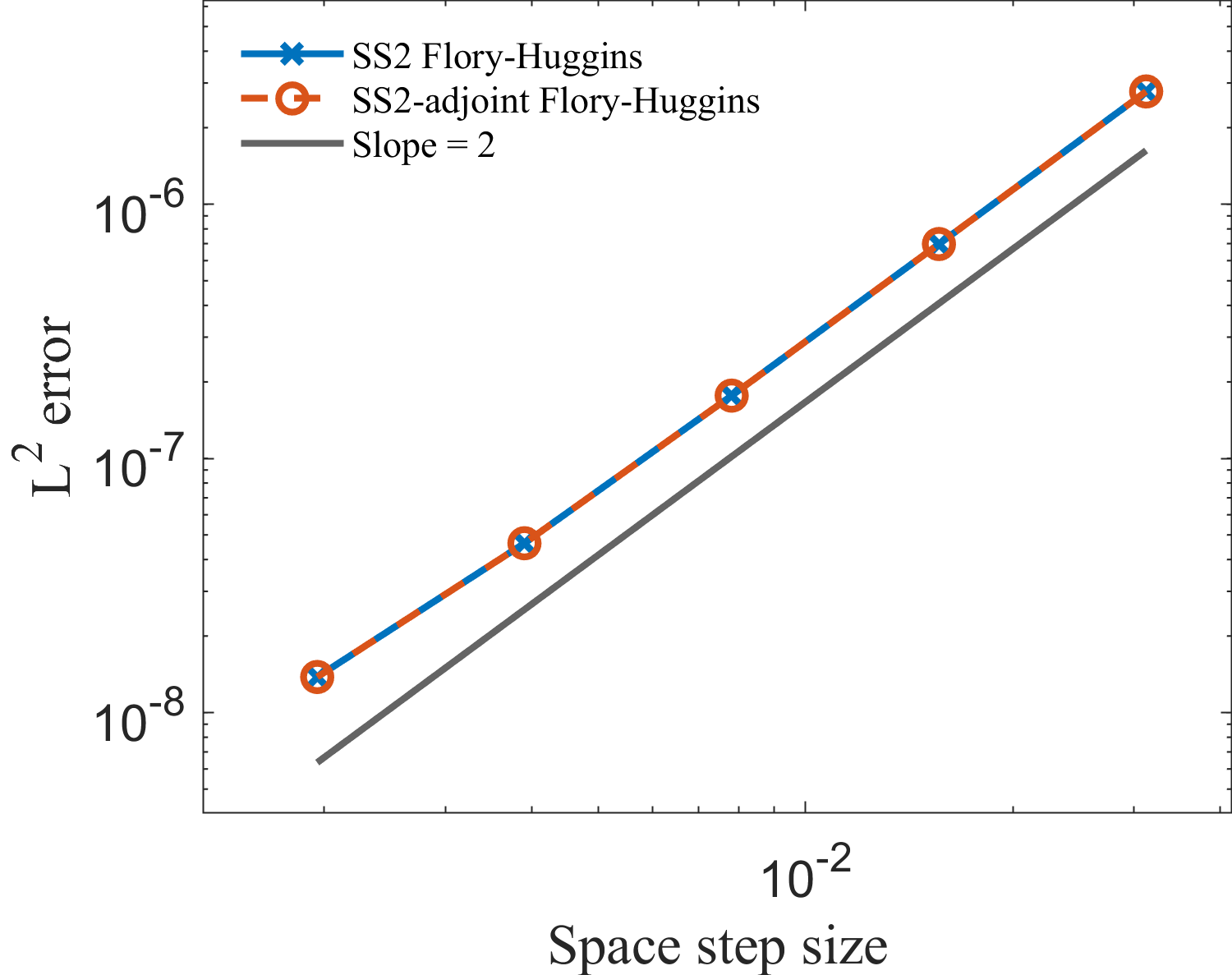}
	\caption{ \label{sinesine_spatial_accuracy} The $L^2$ error vs the space step size computed by the second-order methods for the double-well potential (left) and the Flory-Huggins potential (right).}
\end{figure}
In the spatial convergence test, we also set the final time to $t = 1$ with a fixed time step of $\tau = 2^{-14}$. We use \textbf{SS2} and its adjoint methods to integrate the governing system with respect to the double-well and Flory-Huggins potentials to eliminate temporal error. Afterward, spatial mesh refinement tests are conducted by verifying $h = 2^{-k}$ for $k = 4, 5, \cdots, 10$. Figure \ref{sinesine_spatial_accuracy} depicts the $L^2$ error as a function of the space step size in a logarithmic scale. The desired second-order accuracy in space can be summarized obviously.

To demonstrate the efficiency of our methods, we compare them with the \textbf{SSI1} method in \cite{tang_imex} and the \textbf{CN/AB-Stab} method in \cite{cn_ab} by solving the AC equation with double-well potential. Matlab is used to compute \textbf{SSI1} and \textbf{CN/AB-Stab} methods, whereas our proposed methods are programmed throughout Cython \cite{Cython_article} on the laptop with an eight-core Intel 2.30 GHz Processor and 64 GB Memory.
\begin{rmk}
Due to the extensive use of loops in our algorithm, we opted to write it in Cython to exploit its advantages. Additionally, considering Matlab's superior efficiency in computing FFT compared to Python, we have selected Matlab as the preferred platform for implementing schemes reliant on FFT algorithms.
\end{rmk}

\begin{figure}[H]
	\centering
	\subfigure{
		\includegraphics[width = 0.45\linewidth]{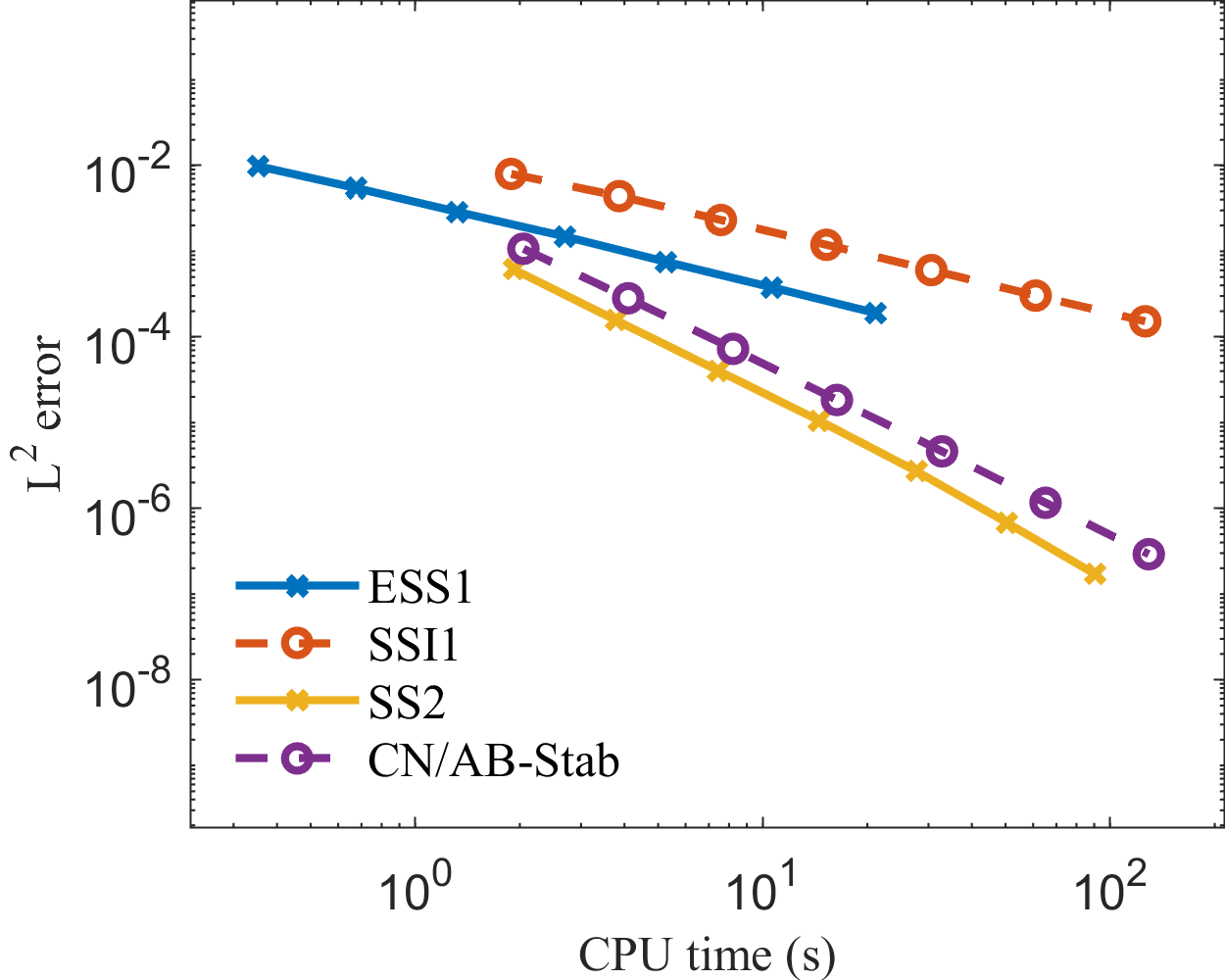}
	}
	\caption{ \label{cpu_error} The $L^2$ error vs CPU times computed by different methods.}
\end{figure}
Figure \ref{cpu_error} summarizes the $L^2$ error versus the CPU times. The test is conducted with a fixed space step size and final time of $h = 1/2048$ and $t = 1$, respectively. We set $\tau = 2^{-k}$, with $k$ varying systematically from $4$ to $10$. It is readily to confirm that the provided \textbf{ESS1} and \textbf{SS2} methods are efficient than \textbf{SSI1} and \textbf{CN/AB-Stab} methods.

\begin{figure}[H]
	\centering
%    \subfigure[\label{sinesine_doublewell_cpu1}]{ 
%		\includegraphics[width = 0.48\linewidth]{cpu_doublewell}
%		%\includegraphics[width = 0.32\linewidth]{cpu_log}
%	}
	 \includegraphics[width = 0.45\linewidth]{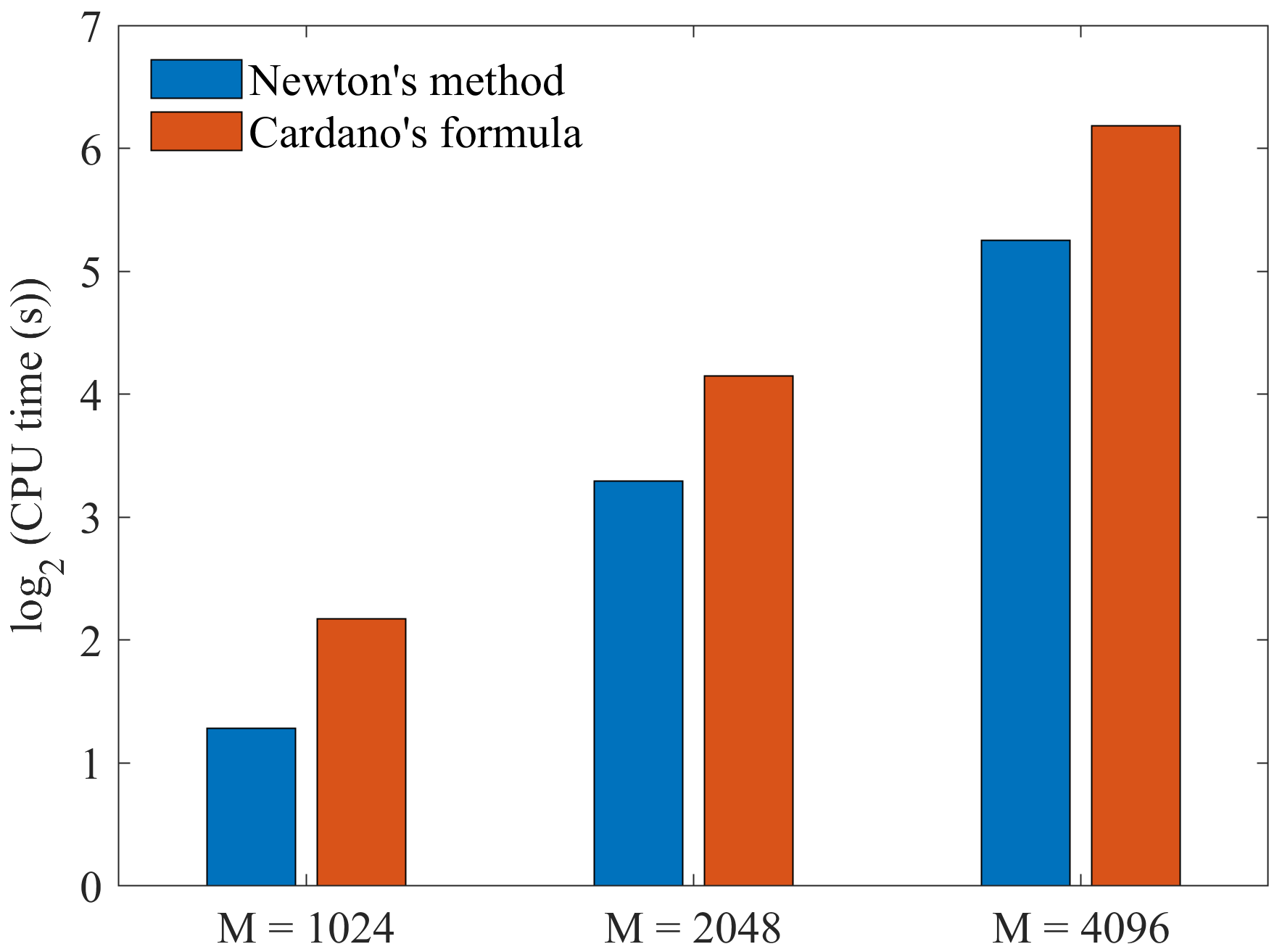}
    \caption{ The CPU time vs the mesh size for the \textbf{ESS1-adjoint} method solved by Newton's method and Cartano's formula.}
    \label{cpu_newton_cardano}
\end{figure}
%Figure \ref{fig_cpu}\subref{sinesine_doublewell_cpu1} displays the base-2 logarithm of CPU time versus $M$. In this test, we fix $\tau = 0.01$ and the termination time $t = 1$. It is clear that when $M$ is small, the fully implicit \textbf{ESS1-adjoint}, \textbf{SS2}, and \textbf{SS2-adjoint} methods are less efficient than the other methods. Nevertheless, when $M$ increases, the efficiency of \textbf{SSI1} and \textbf{CN/AB-Stab} methods decreases. Specifically, when $M = 4096$, all of the proposed methods, even the fully implicit ones, are more efficient than \textbf{SSI1} and \textbf{CN/AB-Stab} methods.
 We explore the efficiency problem discussed in Remark \ref{nonlinear_double_well} by fixing the time step to $\tau = 10^{-2}$ and solving the nonlinear system for the \textbf{ESS1-adjoint} method until $t = 1$ using Newton's method and Cardano's formula, respectively. In Figure \ref{cpu_newton_cardano}, the computational cost is depicted as a function of $M$ on a logarithmic scale. It appears that Newton's method is more effective than Cardano's formula. 
%We observe that the functions \texttt{cbrt} and \texttt{sqrt} are imported from the existing routine \texttt{libc.math} to compute the square and cubic roots in \eqref{root_cubic}.

%\begin{figure}[H]
%	\centering
%	\includegraphics[width = 0.32\linewidth]{cpu_newton_cardano}
%	\caption{\label{fig_cpu_newton} The CPU time vs the mesh size computed by the \textbf{ESS1-adjoint} method, with the nonlinear system solved by the Newton's method and the Cardano's formula.}
%\end{figure}
We solve the AC equation with $\Omega = (0, 2\pi)^2$ and $\varepsilon = 0.05$ in \eqref{ac} to confirm the DEDL and DMP of the proposed methods. The following initial value consists of eight circles, whose centers and radii are specified in Table \ref{eight_center_radii}.
\begin{equation}\label{eight_initial}
	u_0(x, y) = - 0.2 + 0.2 \cdot \sum\limits_{i=1}^8 g\left(\sqrt{(x - x_i)^2 + (y-y_i^2)} - r_i\right),
\end{equation}
where
\begin{equation*}
	g(\xi) =
	\left\lbrace
	\begin{aligned}
		 & 2 e^{-\varepsilon^2/\xi^2},  && \text{if} \ \xi < 0, \\
		 & 0,                           && \text{otherwise}.
	\end{aligned}
	\right.
\end{equation*}
\begin{table}[H]
	\begin{center}
		\caption{Centers $(x_i, y_i)$ and radii of circles in \eqref{eight_initial}.}
		\label{eight_center_radii}
		\begin{tabular}{c | c c c c c c c c}
			\hline
			$i$     & $1$     & $2$        & $3$       & $4$       & $5$          & $6$      & $7$        & $8$         \\
			\hline
			$x_i$ & $\pi/2$ & $\pi/4$  & $\pi/2$  & $\pi$   & $49\pi/40$ & $\pi$   & $3\pi/2$ & 5         \\
			$y_i$ & $\pi/2$ & $3\pi/4$ & $5\pi/4$ & $\pi/4$ & $\pi/4$    & $\pi$   & $3\pi/2$ & 3         \\
			$r_i$ & $\pi/5$ & $\pi/10$ & $\pi/10$ & $\pi/8$ & $\pi/8$    & $\pi/4$ & $\pi/4$  & $2\pi/15$ \\
			\hline
		\end{tabular}
	\end{center}
\end{table}
We employ a variety of different methods to integrate the AC equation up to $t = 30$ by specifying $h = \pi/256$ and $\tau = 0.01$, so as to satisfy the step size requirements in the theoretical analysis. 
\begin{figure}[H]
	\centering
	\subfigure{
		\includegraphics[width=0.40\linewidth]{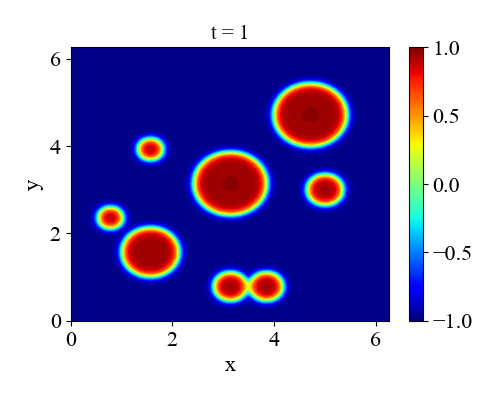}
		\includegraphics[width=0.40\linewidth]{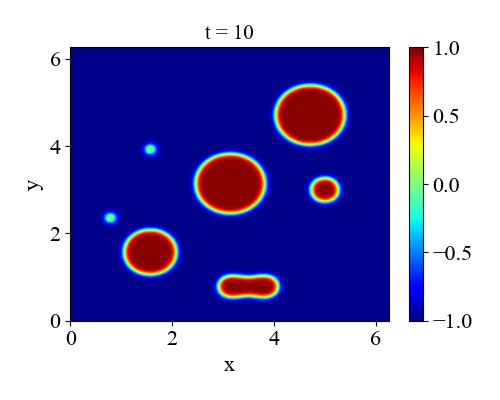}
	}
	\subfigure{
		\includegraphics[width=0.40\linewidth]{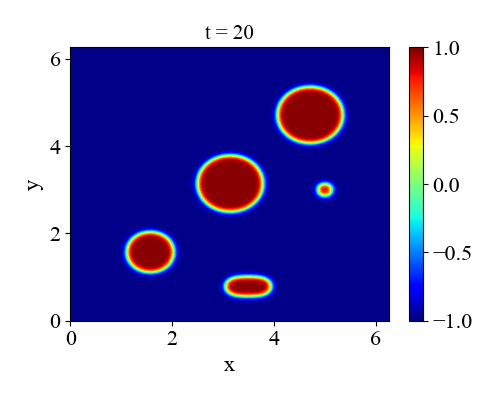}
		\includegraphics[width=0.40\linewidth]{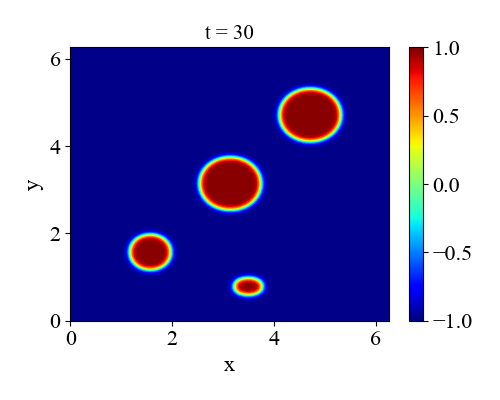}
	}
	\caption{ \label{phase_field_eight_circles} Numerical solutions of the eight circles example for the AC equation with double-well potential solved by the \textbf{SS2}.}
\end{figure}
\begin{figure}[H]
	\centering
	\subfigure{
		\includegraphics[width=0.45\linewidth]{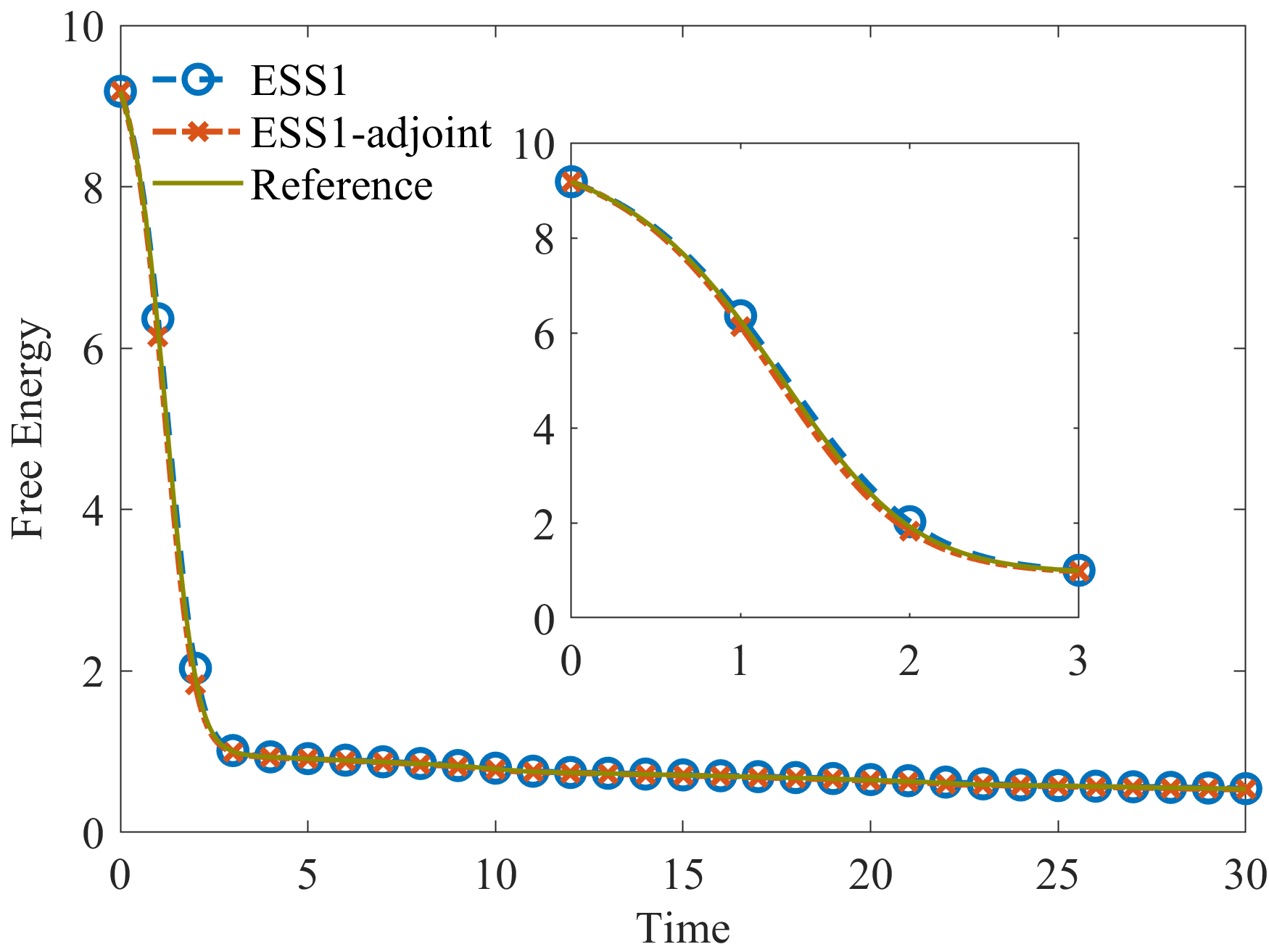}\quad
		\includegraphics[width=0.45\linewidth]{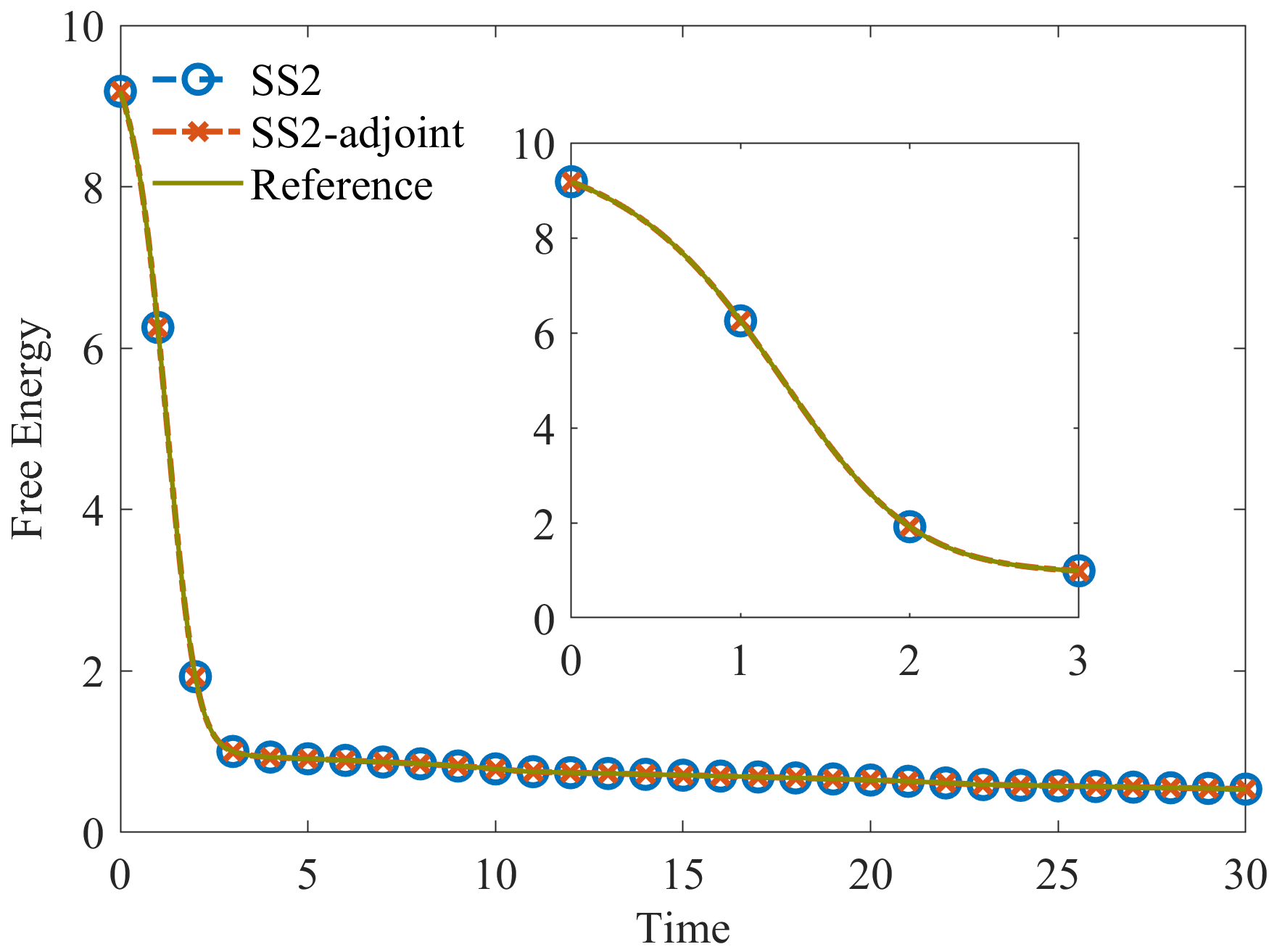}
	}

	\subfigure{
		\includegraphics[width=0.45\linewidth]{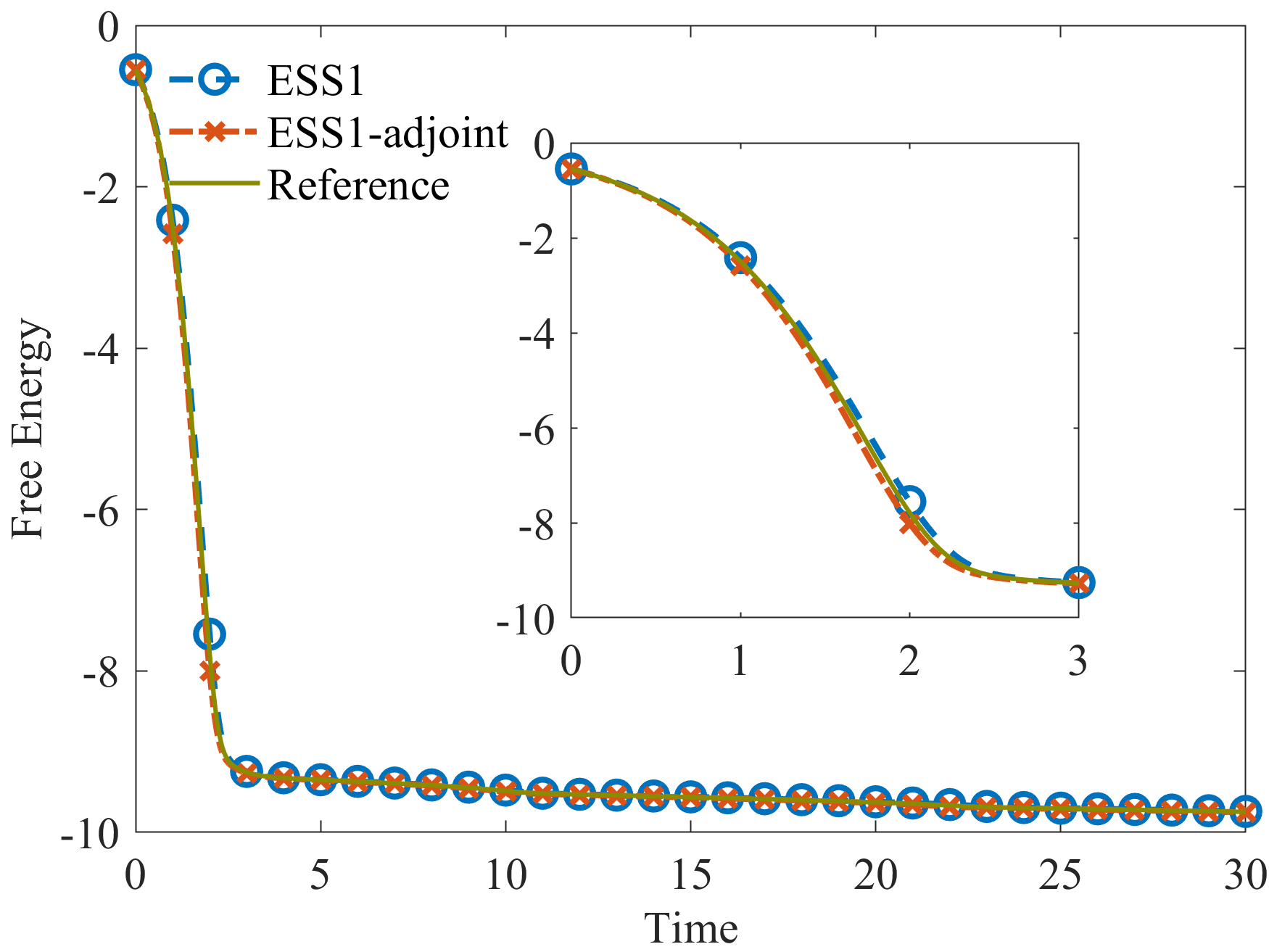}\quad
		\includegraphics[width=0.45\linewidth]{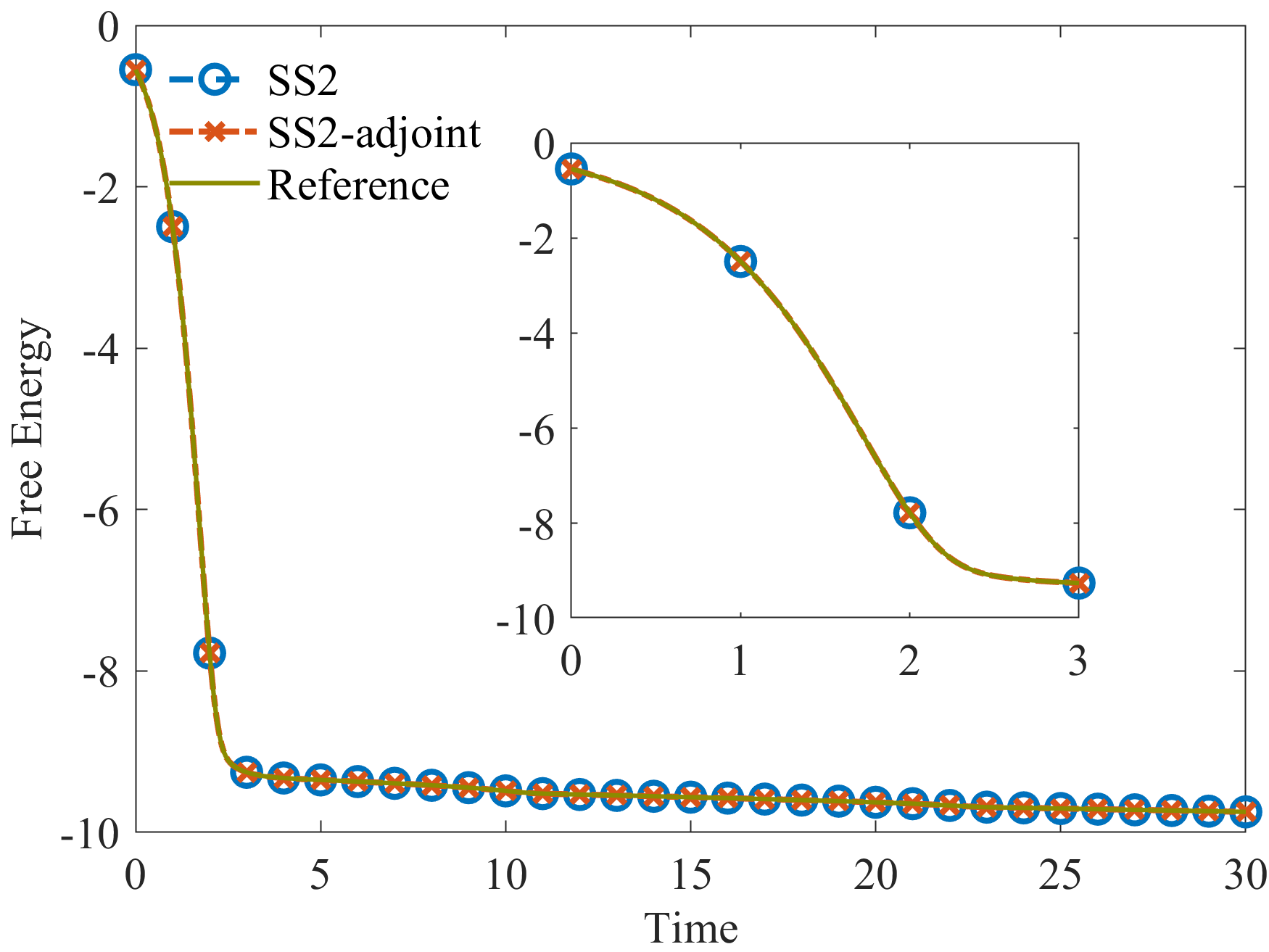}
	}
	\caption{ \label{energy_eight_circles} Evolution of the free energy solved by different methods for the AC equation with double-well potential (the first row) and the Flory-Huggins potential (the second row).}
\end{figure}
\begin{figure}[H]
	\centering
	\subfigure{
		\includegraphics[width=0.45\linewidth]{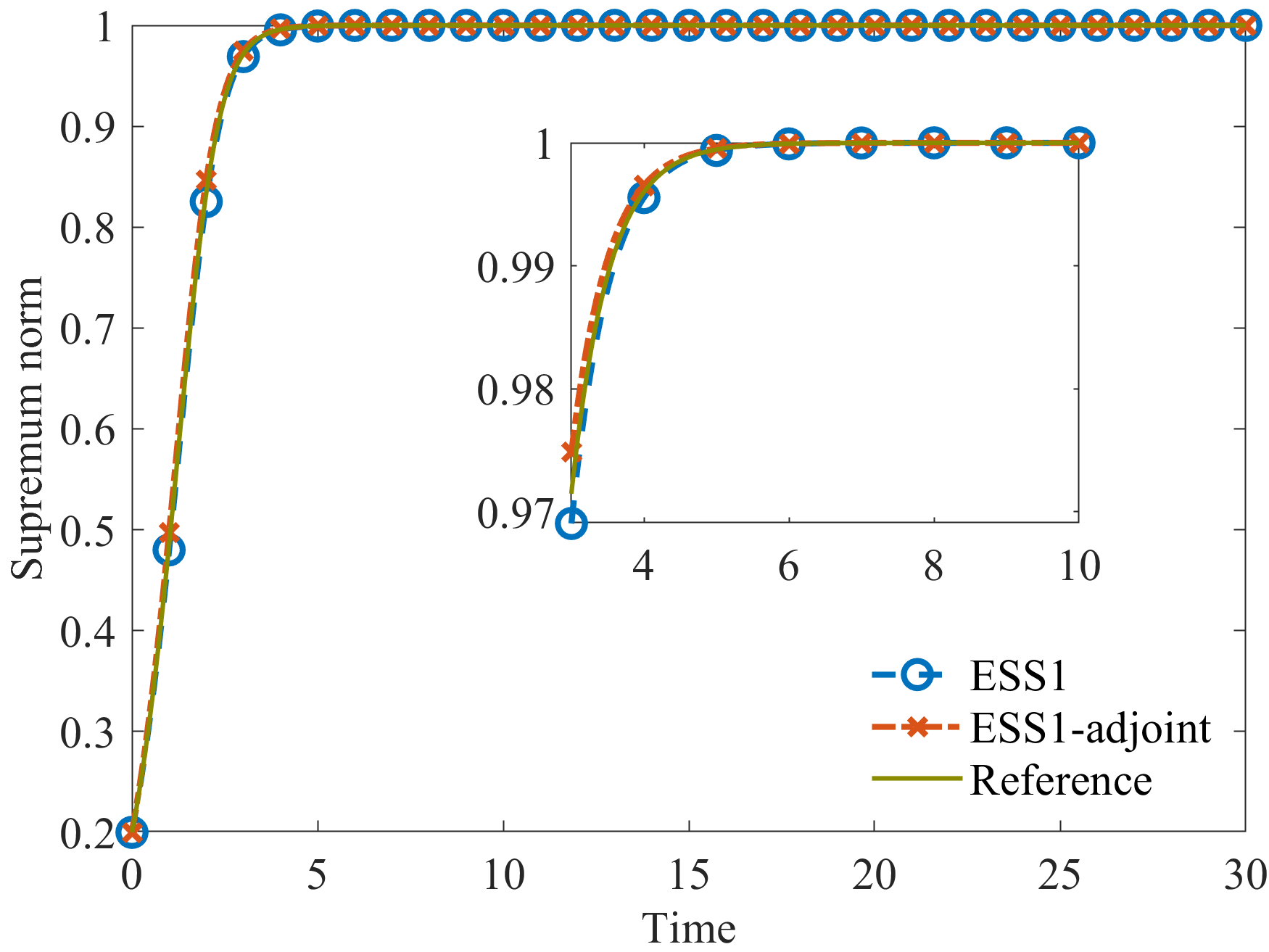}\quad
		\includegraphics[width=0.45\linewidth]{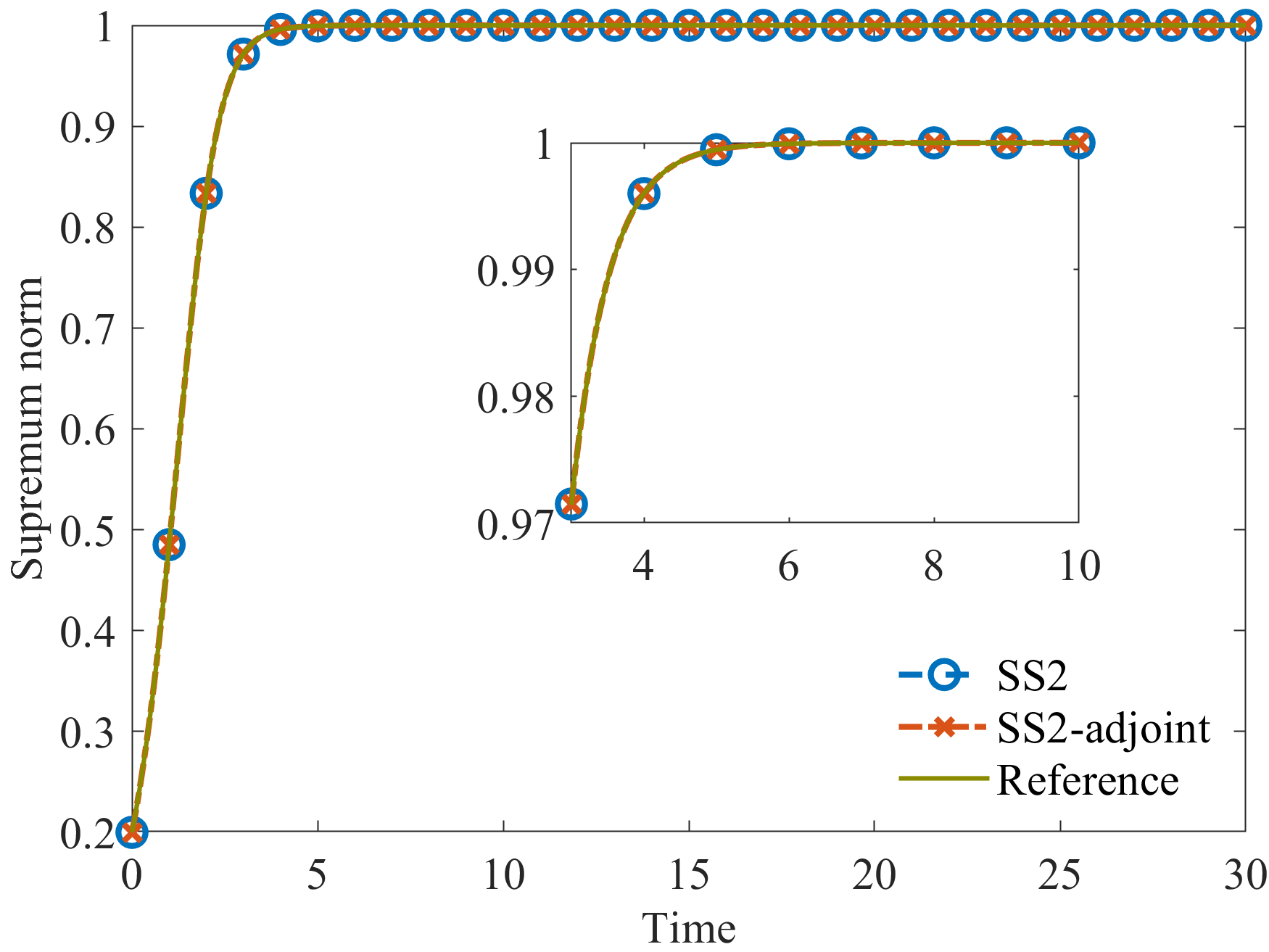}
	}

	\subfigure{
		\includegraphics[width=0.45\linewidth]{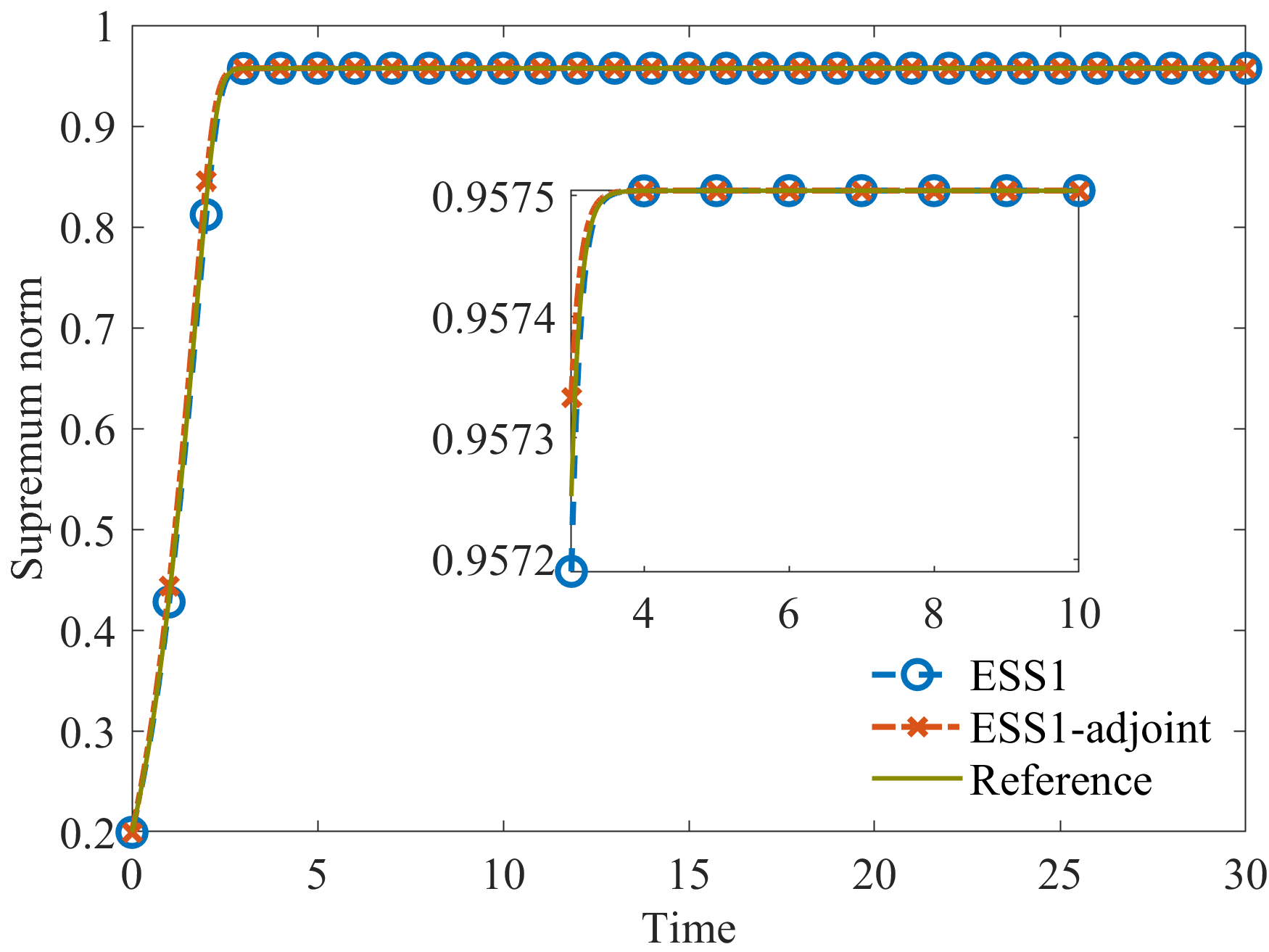}\quad
		\includegraphics[width=0.45\linewidth]{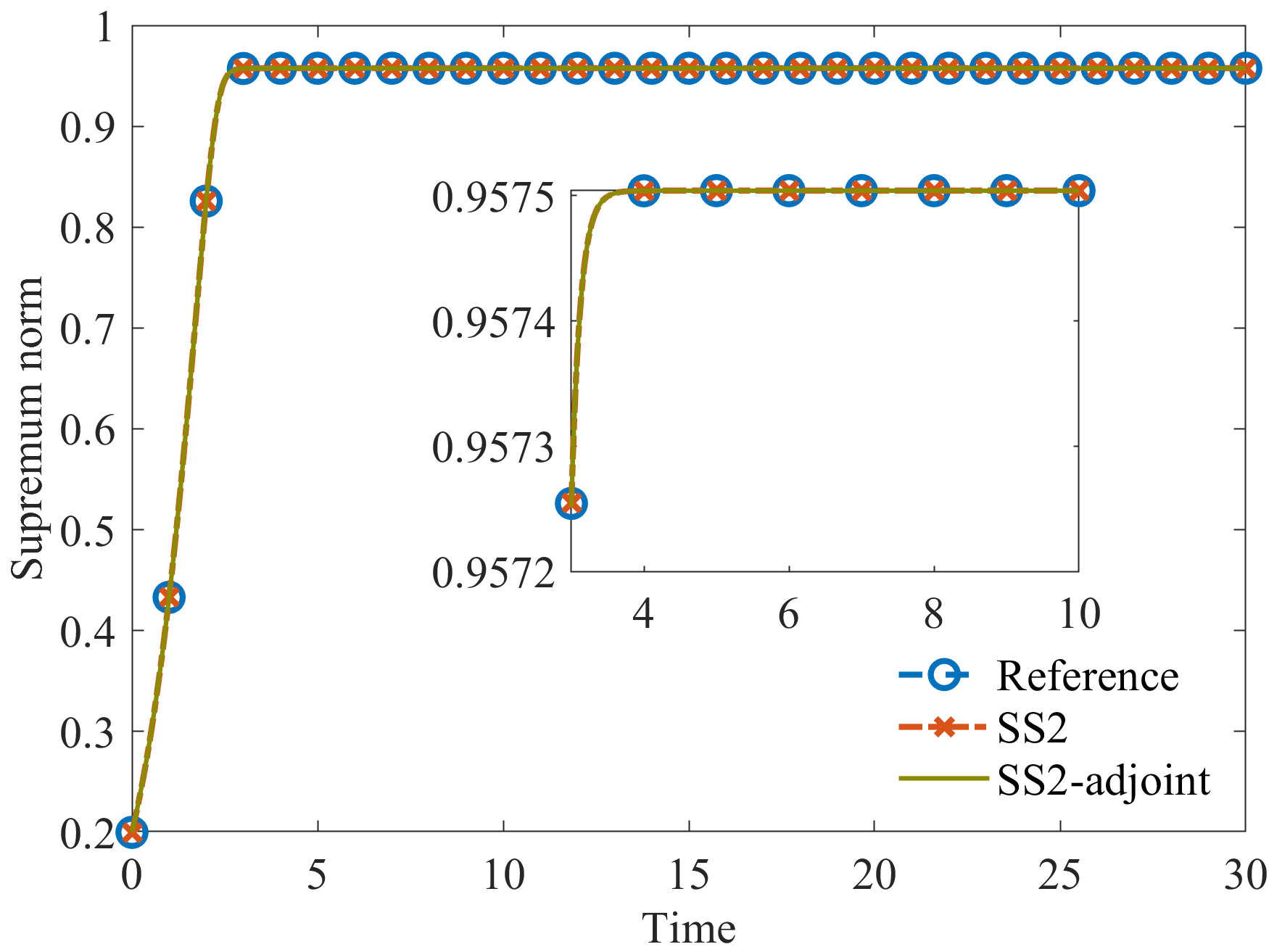}
	}
	\caption{ \label{supremum_norm_eight_circles} Evolution of the supermum norm solved by different methods for the AC equation with double-well potential (the first row) and the Flory-Huggins potential (the second row).}
\end{figure}

The evolution of the phase field is depicted in Figure \ref{phase_field_eight_circles}, where the fusion and the annihilation of the circles take place gradually over time. Figure \ref{energy_eight_circles} displays the evolution of the free energy over time. Throughout all of the provided tests, it is observable that the free energy decreases monotonically. Moreover, the free energy related to the Flory-Huggins potential differs from that associated with the double-well potential. The time history of the discrete supremum norm of the solutions is demonstrated in Figure \ref{supremum_norm_eight_circles}, which is always less than $1$ in the case of double-well potential and does not exceed $0.9575$ in the case of Flory-Huggins potential.

\section{Conclusion}
Based on the stabilized techniques, and the Saul'yev methods, we construct a novel class of efficient methods for the AC equation. All of them can be solved componentwisely. At first, we present a first-order \textbf{ESS1} method, in composition with its adjoint method, the second-order \textbf{SS2}, as well as its adjoint method are constructed. Both \textbf{ESS1} and its adjoint methods have been proved to be energy stable and DMP preserving. Consequently, their composition methods also preserve these two properties. The presented methods are subjected to rigorous analysis for solvability, consistency, and convergence. Numerical experiments are performed to confirm these analysis and demonstrate the superior advantages of the provided methods in efficiency.

\section{Acknowledements}
This work is supported by the National Key Research and Development
Project of China (2018YFC1504205), the National Natural Science Foundation of China (12171245, 11971242).


\begin{thebibliography}{10}
	
	\bibitem{Cython_article}
	S.~Behnel, R.~W. Bradshaw, C.~Citro, L.~Dalcin, D.~S. Seljebotn, and K.~Smith.
	\newblock Cython: The best of both worlds.
	\newblock {\em Comput. Sci. Eng.}, 13:31--39, 2011.
	
	\bibitem{lag}
	Q.~Cheng, C.~Liu, and J.~Shen.
	\newblock {A new Lagrange multiplier approach for gradient flows}.
	\newblock {\em Comput. Methods Appl. Mech. Engrg.}, 367:113030, 2020.
	
	\bibitem{gsav}
	Q.~Cheng, C.~Liu, and J.~Shen.
	\newblock {Generalized SAV approaches for gradient systems}.
	\newblock {\em J. Comput. Appl. Math}, 394:113532, 2021.
	
	\bibitem{du_2019}
	Q.~Du, L.~Ju, X.~Li, and Z.~Qiao.
	\newblock {Maximum principle preserving exponential time differencing schemes
		for the nonlocal Allen-Cahn equation}.
	\newblock {\em SIAM J. Numer. Anal.}, 57:875–898, 2019.
	
	\bibitem{du_2021}
	Q.~Du, L.~Ju, X.~Li, and Z.~Qiao.
	\newblock {Maximum bound principles for a class of semilinear parabolic
		equations and exponential time-differencing schemes}.
	\newblock {\em SIAM Rev.}, 63:317--359, 2021.
	
	\bibitem{convex_splitting1}
	C.~Elliot and A.~Stuart.
	\newblock {The global dynamics of discrete semilinear parabolic equations}.
	\newblock {\em SIAM J. Numer. Anal.}, 30:1622--1663, 1993.
	
	\bibitem{eyre_1998}
	D.~J. Eyre.
	\newblock {Unconditionally gradient stable time marching the Cahn-Hilliard
		equations}.
	\newblock {\em Mater. Res. Soc. Sympos. Proc.}, 529:39--46, 1998.
	
	\bibitem{cn_ab}
	X.~Feng, T.~Tang, and J.~Yang.
	\newblock {Stabilized Crank-Nicolson/Adams-Bashforth schemes for phase field
		models}.
	\newblock {\em East Asian J. Appl. Math.}, 3(1):59--80, 2013.
	
	\bibitem{002}
	P.~J. Flory.
	\newblock {Thermodynamics of high polymer solutions}.
	\newblock {\em J. Chem. Phys.}, 10(1):56--61, 1942.
	
	\bibitem{dvd_book}
	D.~Furihata and T.~Matsuo.
	\newblock {\em Discrete Variational Derivative Method. A Structure-Preserving
		Numerical Method for Partial Differential Equations}.
	\newblock Chapman and Hall/CRC, 1st edition, 2011.
	
	\bibitem{svm}
	Y.~Gong, Q.~Hong, and Q.~Wang.
	\newblock {Supplementary variable method for thermodynamically consistent
		partial differential equations}.
	\newblock {\em Comput. Methods Appl. Mech. Engrg.}, 381:113746, 2021.
	
	\bibitem{ieq_gong}
	Y.~Gong, J.~Zhao, and Q.~Wang.
	\newblock {Arbitrarily high-order linear energy stable schemes for gradient
		flow models}.
	\newblock {\em J. Comput. Phys.}, 419:109610, 2020.
	
	\bibitem{ieq1}
	F.~Guillén-González and G.~Tierra.
	\newblock {On linear schemes for a Cahn–Hilliard diffuse interface model}.
	\newblock {\em J. Comput. Phys.}, 234:140--171, 2013.
	
	\bibitem{hairer_book}
	E.~Hairer, C.~Lubich, and G.~Wanner.
	\newblock {\em Geometric Numerical Integration: Structure-Preserving Algorithms
		for Ordinary Differential Equations}.
	\newblock Springer-Verlag, Berlin, 2nd edition, 2006.
	
	\bibitem{hou_cnab}
	T.~Hou and H.~Leng.
	\newblock {Numerical analysis of a stabilized Crank-Nicolson/Adams-Bashforth
		finite differ- ence scheme for Allen-Cahn equations}.
	\newblock {\em Appl. Math. Lett.}, 102:106150, 2020.
	
	\bibitem{hou_leapfrog}
	T.~Hou, D.~Xiu, and W.~Jiang.
	\newblock {A new second-order maximum-principle preserving finite difference
		scheme for Allen-Cahn equations with periodic boundary conditions}, 2020.
	
	\bibitem{003}
	M.~L. Huggins.
	\newblock {Solutions of long cain compounds}.
	\newblock {\em J. Chem. Phys.}, 9(5):440, 1941.
	
	\bibitem{esav_high}
	C.~Jiang, J.~Cui, X.~Qian, and S.~Song.
	\newblock {High-order linearly implicit structure-preserving exponential
		integrators for the nonlinear Schrödinger equation}.
	\newblock {\em J. Sci. Comput.}, 90:27, 2020.
	
	\bibitem{esav_kg}
	C.~Jiang, Y.~Wang, and W.~Cai.
	\newblock {A linearly implicit energy-preserving exponential integrator for the
		nonlinear Klein-Gordon equation}.
	\newblock {\em J. Comput. Phys.}, 419:18, 2020.
	
	\bibitem{ju_ifrk_high}
	L.~Ju, X.~Li, and Z.~Qiao.
	\newblock {Maximum bound principle preserving integrating factor Runge-Kutta
		methods for semilinear parabolic equations}.
	\newblock {\em J. Comput. Phys.}, 439:110405, 2021.
	
	\bibitem{ju_gsav}
	L.~Ju, X.~Li, and Z.~Qiao.
	\newblock {Generalized SAV-exponential integrator schemes for Allen-Cahn type
		gradient flows}.
	\newblock {\em SIAM J. Numer. Anal.}, 60(4):1905--1931, 2022.
	
	\bibitem{ju_esav}
	L.~Ju, X.~Li, and Z.~Qiao.
	\newblock {Stabilized Exponential-SAV Schemes Preserving Energy Dissipation Law
		and Maximum Bound Principle for The Allen–Cahn Type Equations}.
	\newblock {\em J. Sci. Comput.}, 92, 2022.
	
	\bibitem{li_strang}
	D.~Li, C.~Quan, and J.~Xu.
	\newblock {Stability and convergence of Strang splitting. Part I: Scalar
		Allen-Cahn equation}.
	\newblock {\em J. Comput. Phy.}, 458:111087, 2022.
	
	\bibitem{sav_li}
	D.~Li and W.~Sun.
	\newblock {Linearly implicit and high-order energy-conserving schemes for
		nonlinear wave equations}.
	\newblock {\em J. Sci. Comput.}, 83:17, 2020.
	
	\bibitem{sav_nlsw}
	X.~Li, Y.~Gong, and L.~Zhang.
	\newblock {Linear high-order energy-preserving schemes for the nonlinear
		Schrödinger equation with wave operator using the scalar auxiliary variable
		approach}.
	\newblock {\em J. Sci. Comput.}, 88:25, 2021.
	
	\bibitem{liao_siam1}
	H.~Liao, T.~Tang, and T.~Zhou.
	\newblock {On energy stable, maximum-principle preserving, second-order BDF
		scheme with variable steps for the Allen-Cahn equation}.
	\newblock {\em SIAM J. Numer. Anal.}, 58(4):2294--2314, 2020.
	
	\bibitem{if_nac}
	C.~Nan and H.~Song.
	\newblock {The high-order maximum-principle-preserving integrating factor
		Runge-Kutta methods for nonlocal Allen-Cahn equation}.
	\newblock {\em J. Comput. Phys.}, 456:111028, 2022.
	
	\bibitem{saulyev_1}
	V.~K. Saul'yev.
	\newblock {On a method of numerical integration of a diffusion equation}.
	\newblock {\em Dokl Akad Nauk SSSR(in Russian)}, 115:1077--1079, 1957.
	
	\bibitem{sav_shen}
	J.~Shen, J.~Xu, and J.~Yang.
	\newblock {The scalar auxiliary variable (SAV) approach for gradient flows}.
	\newblock {\em J. Comput. Phys.}, 353:407–416, 2018.
	
	\bibitem{sav_shen_siam}
	J.~Shen, J.~Xu, and J.~Yang.
	\newblock {A new class of efficient and robust energy stable schemes for
		gradient flows}.
	\newblock {\em SIAM Rev.}, 61:474--506, 2019.
	
	\bibitem{shen_2010}
	J.~Shen and X.~Yang.
	\newblock {Numerical approximations of Allen-Cahn and Cahn-Hilliard equations}.
	\newblock {\em Discrete Contin. Dyn. Syst.}, 28(4):1669--1691, 2010.
	
	\bibitem{csrk}
	J.~Shin, H.~G. Lee, and J.-Y. Lee.
	\newblock {Unconditionally stable methods for gradient flow using convex
		splitting Runge-Kutta scheme}.
	\newblock {\em J. Comput. Phys.}, 347:367–381, 2017.
	
	\bibitem{tang_imex}
	T.~Tang and J.~Yang.
	\newblock {Implicit-explicit scheme for the Allen-Cahn equation preserves the
		maximum principle}.
	\newblock {\em J. Comput. Math.}, 34:471–481, 2016.
	
	\bibitem{thomas}
	L.~Thomas.
	\newblock {\em Elliptic Problems in Linear Difference Equations over a
		Network}.
	\newblock Watson Sc. Comp. Lab. Rep., Columbia University, New York., 1949.
	
	\bibitem{xiao_split}
	X.~Xiao and X.~Feng.
	\newblock {A second-order maximum bound principle preserving operator splitting
		method for the Allen-Cahn equation with applications in multi-phase systems}.
	\newblock {\em Math. Comput. Simulation}, 202:36--58, 2022.
	
	\bibitem{xiao_fd}
	X.~Xiao, X.~Feng, and J.~Yuan.
	\newblock { The stabilized semi-implicit finite element method for the surface
		Allen-Cahn equation}.
	\newblock {\em Discrete Contin. Dyn. Syst. Ser. B}, 22:2857--2877, 2017.
	
	\bibitem{mbp_fd}
	J.~Yang, Q.~Du, and W.~Zhang.
	\newblock {Uniform $L^p$-bound of the Allen-Cahn equation and its numerical
		discretization}.
	\newblock {\em Int. J. Numer. Anal. Model.}, 15:213--227, 2018.
	
	\bibitem{ieq3}
	X.~Yang.
	\newblock {X. Yang, Linear, first and second-order, unconditionally energy
		stable numerical schemes for the phase field model of homopolymer blends}.
	\newblock {\em J. Comput. Phys.}, 327:294--316, 2016.
	
	\bibitem{ieq4}
	X.~Yang and L.~Ju.
	\newblock { Efficient linear schemes with unconditional energy stability for
		the phase field elastic bending energy model}.
	\newblock {\em Comput. Methods Appl. Mech. Eng.}, 135:691--712, 2017.
	
	\bibitem{ieq2}
	X.~Yang, J.~Zhao, Q.~Wang, and J.~Shen.
	\newblock {Numerical approximations for a three components Cahn–Hilliard
		phase-field model based on the invariant energy quadratization method}.
	\newblock {\em Math. Models Methods Appl. Sci.}, 27(11):1993--2030, 2017.
	
	\bibitem{zhang_anm}
	H.~Zhang, J.~Yan, X.~Qian, and S.~Song.
	\newblock {Numerical analysis and applications of explicit high order maximum
		principle preserving integrating factor Runge-Kutta schemes for Allen-Cahn
		equation}.
	\newblock {\em Appl. Numer. Math.}, 161:372--390, 2021.
	
\end{thebibliography}
\end{document}